\documentclass[english]{amsart}
\usepackage[T1]{fontenc}
\usepackage{euscript}
\usepackage{textcomp}
\usepackage{amsmath}
\usepackage{amssymb}
\usepackage{amsthm}
\usepackage[english]{babel}
\usepackage{graphicx}
\usepackage{caption}
\usepackage{subcaption}
\theoremstyle{definition}
\newtheorem{definition}{Definition}
\newtheorem{theorem}{Theorem}
\newtheorem{lemma}[theorem]{Lemma}
\newtheorem{proposition}[theorem]{Proposition}
\newtheorem{corollary}[theorem]{Corollary}
\newtheorem{conjecture}[theorem]{Conjecture}
\newtheorem{remark}{Remark}

\title{Diffusion for chaotic plane sections of 3-periodic surfaces}
\author {Artur Avila}
\address{CNRS UMR 7586, Institut de Math\'ematiques de Jussieu - Paris Rive Gauche, Batiment Sophie Germain, Case 7021, 75205 Paris Cedex 13, France\\
and IMPA, Estrada Dona Castarino 110, 22460-320, Rio de Janeiro, Brazil}
\email{artur@math.jussieu.fr}
\author{Pascal Hubert}
\address{Institut de Math\'ematiques de Marseille, 39 rue F. Joliot-Curie, 13453 Marseille Cedex 20, France}
\email{pascal.hubert@univ-amu.fr}
\author {Alexandra Skripchenko}
\address{Faculty of Mathematics, National Research University Higher School of Economics, Vavilova St. 7, 112312 Moscow, Russia}
\email{sashaskrip@gmail.com}

\begin{document}
\begin{abstract}
We study chaotic plane sections of some particular family of triply periodic surfaces. 
The question about possible behavior of such sections was posed by S. P. Novikov. 
We prove some estimations on the diffusion rate of these sections using the connection between Novikov's problem and systems of isometries - some natural generalization of interval exchange transformations. Using thermodynamical formalism, we construct an invariant measure for systems of isometries of a special class called the Rauzy gasket, and investigate the main properties of the Lyapunov spectrum of the corresponding suspension flow. 
\end{abstract}

\maketitle
\section{Introduction}
\subsection{Historical background and main result}
A surface $\hat M$ in $\mathbb R^{3}$ is called \emph{triply periodic} if it is invariant under translations on vectors from the integral lattice $\mathbb Z^{3}$ (see Figure \ref{sur} for an example). In the general context the problem of the asymptotic behavior of sections of triply periodic surfaces by planes of some fixed direction was posed by S.P. Novikov in connection with the theory of metals (see \cite{N}). 

Sections of $\hat M$ by the family of planes of fixed direction $H$ consist of the collection of curves (see Figure \ref{sec}). It was shown (see \cite{Zo} and \cite{D1}) that usually these curves contain either only closed components (\emph{trivial case}) or every unbounded component has a form of a finitely deformed straight line and therefore has a strong asymptotic direction (\emph{integrable case}). However, in \cite{D1} I. Dynnikov proved that it is possible that the section does not have such a direction and wanders around the whole plane (\emph{chaotic case}). See Figure \ref{sec} for an example of chaotic section.



\begin{figure}
    \vspace{-2.5cm}
    \centering
    \begin{subfigure}[b]{0.5\textwidth}
     
        \centering
        \includegraphics[width=\textwidth]{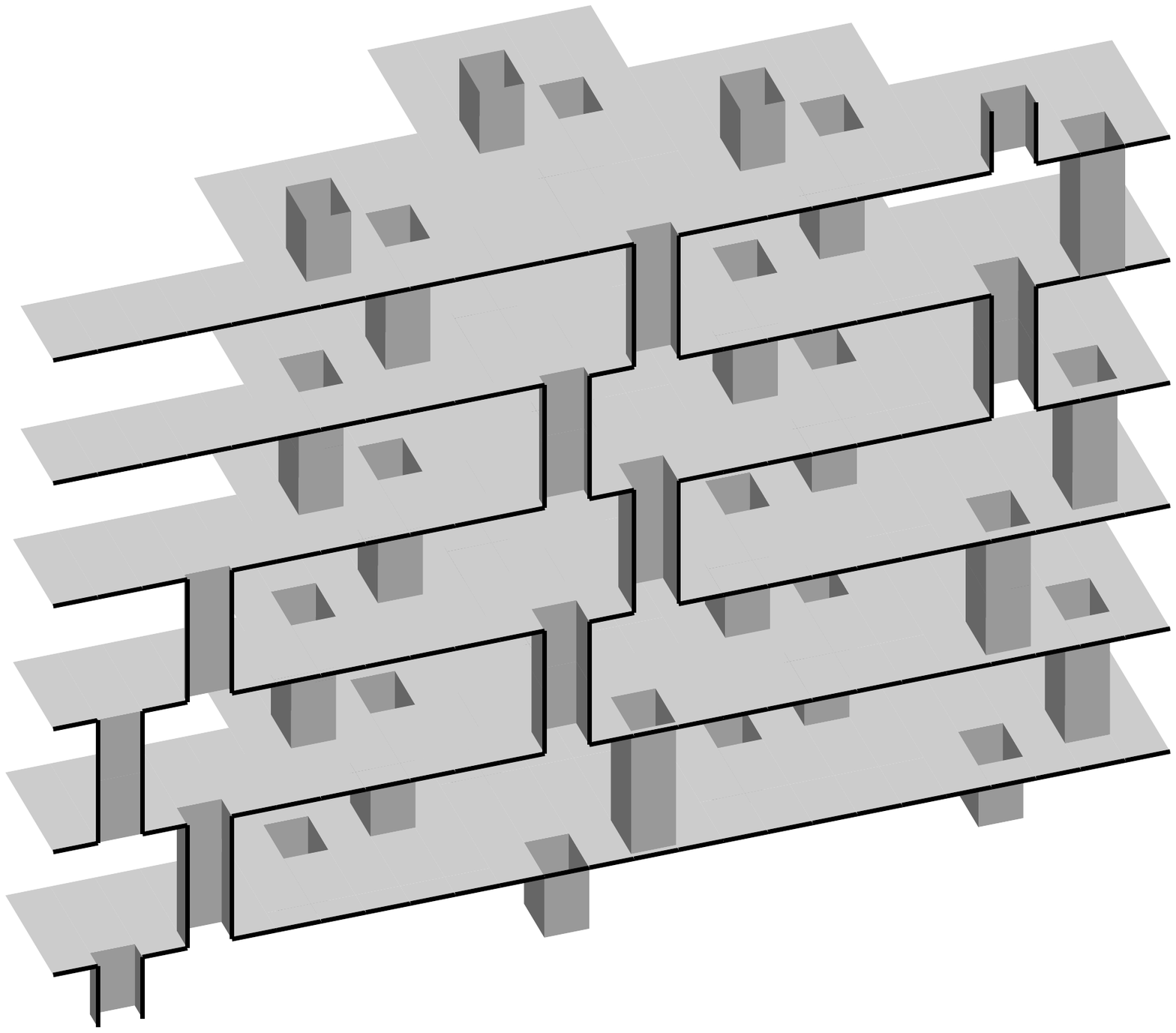}  
        \vspace{-3.7cm}   
        \label{f}
    \end{subfigure}
    \hfill
    \begin{subfigure}[b]{0.75\textwidth}
     \vspace{-2cm}
        \centering
        \includegraphics[width=\textwidth]{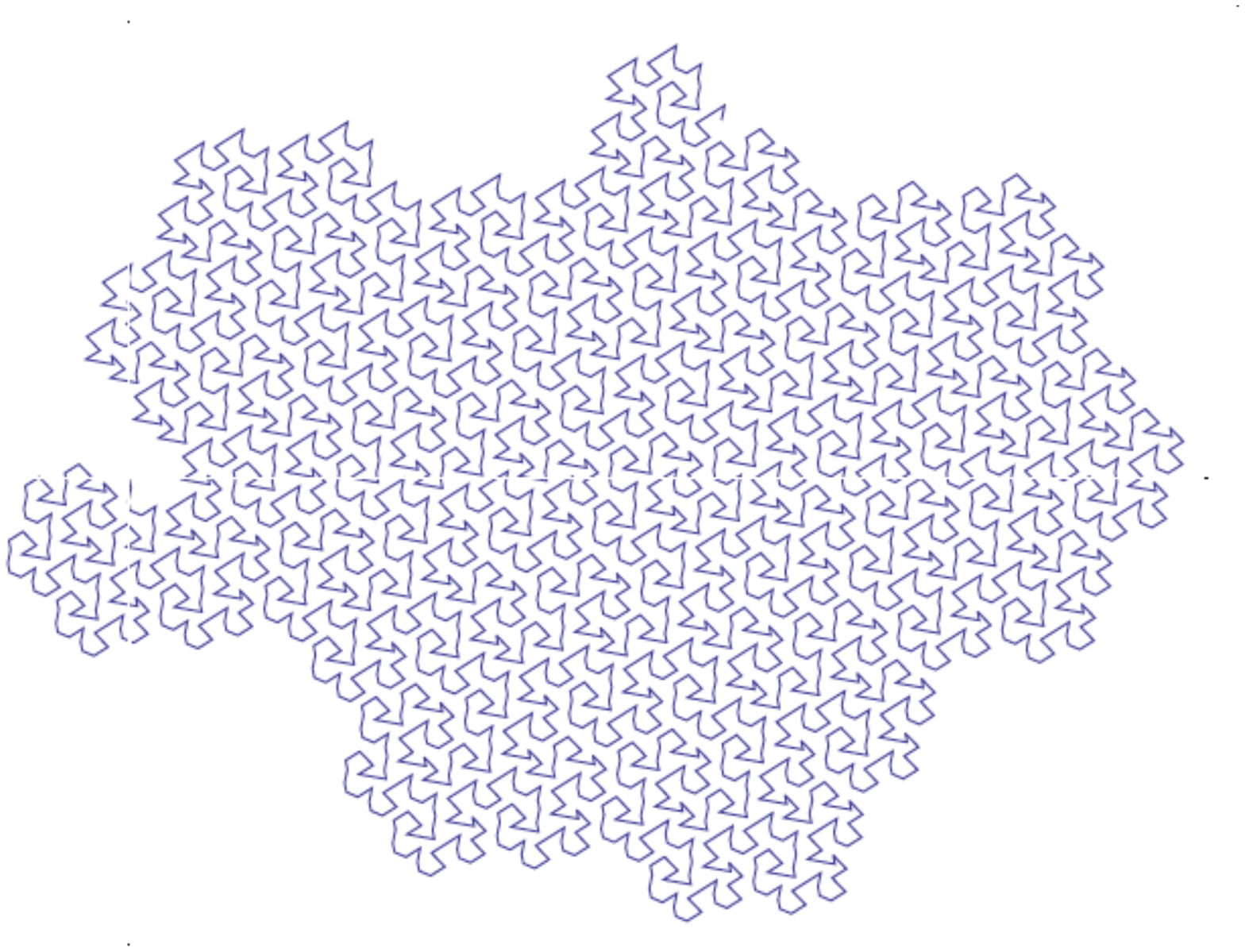}
        \put(-247,215){\vector(1,-1){20}}
        \put(-251,217){$x$}
        \put(-91,282.7){\vector(-1,-1){20}}
         \put(-92,285){$x_t$}
         \vspace{-1cm} 
        \label{}
    \end{subfigure}
  \vspace{-2.5cm}  
  \caption{Chaotic plane section of 3-periodic surfaces.}
   \label{sec}
\end{figure}

In the current paper we examine chaotic sections for the particular family of surfaces of genus $3$ and foliations on them with $2$ double saddles. The idea of our main theorem is the following: in \cite{D} Dynnikov introduced a way to construct chaotic regimes from some natural generalization of interval exchange transformations (IET) called \emph{systems of isometries}. More precisely, he found a bijection between chaotic regimes and systems of isometries of order $3$ of thin type. In \cite{AHS} we studied a particular family of these systems and the corresponding parameter set called the Rauzy gasket (the object also appeared in \cite{AS}, \cite{L}). In all these papers, although the motivation was completely different, the dynamics came from the iterated application of the same algorithm that will be described below, and the Rauzy gasket was an attractor of the dynamical system associated with this generating algorithm.   

An application of Dynnikov's construction to the family of systems of isometries mentioned above provides a special collection of chaotic regimes. Our result describes the typical behavior of these chaotic sections.

\begin{theorem}\label{main}
There exists a probability measure for the Rauzy gasket, and 
the diffusion rate of the trajectories for almost all chaotic regimes (with respect to this measure) is strictly between $\frac{1}{2}$ and $1:$
$$\frac{1}{2}<\limsup_{t\to\infty}\frac{\log d(x,x_{t})}{\log t}<1,$$
where $d(x,y)$ is the standard Euclidian distance between points $x$ and $y$ on the plane, $x$ is some starting point that belongs to the section and $x_t$ is the position of the point after time $t$. 
\end{theorem}

\begin {remark}
The measure we construct is invariant with respect to the generating algorithm mentioned above and it is the measure of maximal entropy for the suspension flow associated with this algorithm. 
\end{remark}

The theorem gives an answer to two questions: existence of an invariant measure for the Rauzy gasket (posed by P. Arnoux and S. Starosta in \cite{AS}) and the diffusion rate of chaotic trajectories (asked by A. Zorich in 2011). 

Morally, our result means that for almost every parameter, the leaves behave in a way which is in some sense intermediate between moving away as slowly as possible (spending in each region the amount of time proportional to the area of region) and as quickly as possible (running away toward infinity with a linear rate). 
A similar result for wind-tree model was established in \cite{DHL}. For our result, as well as for one proved in \cite{DHL}, it is very important that we deal with $\limsup$ (and not with a simple $\lim$): one can check figure \ref{sec}\footnote{See \cite{McM} as a source of the bottom figure} where it is visible that our statement does not hold for the $\lim$.

From the point of view of Novikov's problem, our theorem shows the existence of some leading asymptotic direction for a typical chaotic section.

\subsection{Organization of the paper} 
In section \ref{1} we provide the original statement of Novikov's problem and briefly discuss the connection with the systems of isometries, as well as some open questions related to this problem. We mainly recall the ideas established by Dynnikov in \cite{D}.

In section \ref{2} we present some particular family of systems of isometries (the corresponding set of parameters is called the Rauzy gasket) and describe the associated Markov map and symbolic dynamics.

In section \ref{3} we construct the suspension flow. We also examine some important properties of the roof function. 
 
In section \ref{4} using thermodynamical formalism for countable Markov shift we prove the existence and the uniqueness of the Gibbs measure and the equilibrium measure with respect to the Markov map. With the similar arguments we show existence and uniqueness of the measure of maximal entropy for the suspension flow.
Finally, using Abramov's formula, we obtain a natural invariant measure on the Rauzy gasket.

In section \ref{5} we use the Oseledets theorem to define the Lyapunov exponents for some special cocycle (which is the analogue of Kontsevich-Zorich cocycle over Teichm\"uller flow). Such a cocycle contains the information about the orientation for the band complex that is the suspension of the system of isometries, and this cocycle differs from the one that was used for the definition of the flow. 

In section \ref{6} we express the diffusion rate of the trajectories in the chaotic case in terms of the Lyapunov exponents of the cocycle constructed in the previous section.

In section \ref{7}, we prove some properties of the Lyapunov spectrum for our version of Kontsevich-Zorich cocycle, such as Pisot property and simplicity, and use them to conclude our estimation.  

\subsection{Acknowledgments} We heartily thank A. Zorich for posing the problem and several improvements to the first version of the text. We are very grateful to F. Ledrappier who kindly explained Sarig's theory to us. We also thank I. Dynnikov and V. Delecroix for many fruitful discussions and C. Matheus for his explanations on the Galois version of the twisting/pinching criterium. 

We thank C. McMullen for the bottom part of the Figure \ref{sec}.

We also thank the anonymous referee for many useful suggestions and improvements to the previous version of the paper.

The first author was partially supported by the ERC Starting Grant \textquotedblleft Quasiperiodic\textquotedblright  and by the Balzan project of Jacob Palis. The second author was partially supported by the projet ANR GeoDyM and ANR VALET. The third author was partially supported by the Fondation Sciences Math\'ematiques de Paris, Metchnikov scholarship and the Dynasty Foundation. 

\section{Novikov's problem and systems of isometries}\label{1}
\subsection{General description}
\noindent  Let us start from the formal statement of Novikov's problem posed in \cite{N}.
We consider a triply periodic surface $\hat M$ that is a level surface of some smooth 3-periodic function. The motivation to study asymptotic behavior of regular plane sections of such a surface by the family of parallel planes orthogonal to some non-zero vector $H=(H_{1}, H_{2}, H_{3})$ came from the conductivity theory for monocrystals since the periodic surface can be interpreted as a Fermi surface of some metal and the vector $H$ is the direction of constant magnetic field (see an example of Fermi surface of tin in \cite{Z}). So the plane sections can be seen as electron trajectories in the inverse metal lattice in a presence of a magnetic field.

There exist two equivalent approaches to this problem. In the framework of the first approach, the periodic surface is fixed and different families of planes are considered while using the second approach one fixes the vector $H$ and considers a family of perturbations of a periodic surface. 

Both of these strategies were applied with different results. Using the first one, Zorich in \cite {Zo} proved that if the direction of a plane is a sufficiently small perturbation of a rational direction, then every unbounded component of any nonsingular section goes along a straight line with a bounded deviation from it. With the second one Dynnikov in \cite{D1} generalized this result and proved that typically a regular plane section of a triply periodic surface either consists of compact components only (\emph{trivial case}) or has unbounded components that have the form of finitely deformed periodic family of parallel straight lines (\emph{integrable case}). 

The presence of a strong asymptotic direction of the discussed curves is explained by the fact
that the image of such a curve under the natural projection $\pi: \mathbb R^{3}\rightarrow \mathbb T^{3}=\mathbb R^{3}/\mathbb Z^{3}$ densely fills not the whole surface $M=\pi(\hat M)$ but only a part that has genus one.

\begin{definition}
A plane section of the surface $\hat M$ by a plane is called chaotic if it has at
least one connected component such that the closure of its projection $\pi$ is a subsurface of $M$
(possibly with boundary) of genus strictly greater than one.
\end{definition}

The first example of such a non-typical behavior in which the unbounded components had an
asymptotic direction but did not fit into a strip of finite width was constructed by S. Tsarev in 1992 (see \cite{D1} for details). However, the plane direction in this example is not totally irrational, meaning that the irrationality degree of this vector is $2$. Dynnikov proved that if this condition holds then all regular non-closed section components in the covering space have some asymptotic direction but do not fit into any strip (it means that the trajectories have a form of distorted line but the distortion is not uniformly bounded). He also showed that in a generic situation (when the irrationality degree of $H$ is equal to $3$) the genus of the surface $M$ in the chaotic case should be equal to at least  $3$. Indeed, the construction that we describe later always provides $M$ of genus $3$. 

Finally, Dynnikov proved the following
\begin{theorem} \label{thd} [Dynnikov, 1997]
In the space of pairs $(M,H)$, where M is a null-homologous surface in the $\mathbb T^{3}$ and
$H$ is a covector from $\mathbb R^{3}$, all pairs giving rise to a chaotic foliation $F$ are contained in a subset $\EuScript{R}$ of codimension $1$ and, moreover, form a nowhere dense subset in it.
\end{theorem}

In order to make this statement precise, we have to admit that although the space of all surfaces is infinite-dimensional, the section properties we are interested in depend in the generic case on finitely many parameters, namely, on the positions of the saddles of the foliation $F$ and on the
coordinates of the covector $H$.

The pair $(M,H)$ that gives rise to a chaotic foliation will be called a \emph{chaotic regime}.

Now all remaining open questions in Novikov's problem are related to chaotic case. In particular, we are interested in probability to obtain chaotic section (inside of this special subspace described by Dynnikov) as well as in any information about the geometry of this kind of sections. 
In 2003 Novikov and A. Maltsev (see \cite{MN}) formulated the following
\begin{conjecture}\label{nm}
For every $M$ chaotic regimes form a subset of Hausdorff dimension strictly smaller than $2$ in $\EuScript{R}$.
\end{conjecture}

In \cite{DD} Dynnikov and R. De Leo suggested a particular model to study where they fixed the surface (see Figure \ref{sur}) and varied the family of the plane sections. The conjecture \ref{nm} is proved only for this case (see \cite{AHS}). In our paper we study the same group of examples but use another strategy as mentioned above: we fixed the direction of planes and varied the lattice. 

\begin{figure}
\centering
\vspace{-2cm}
\includegraphics[height=8cm, width=6cm]{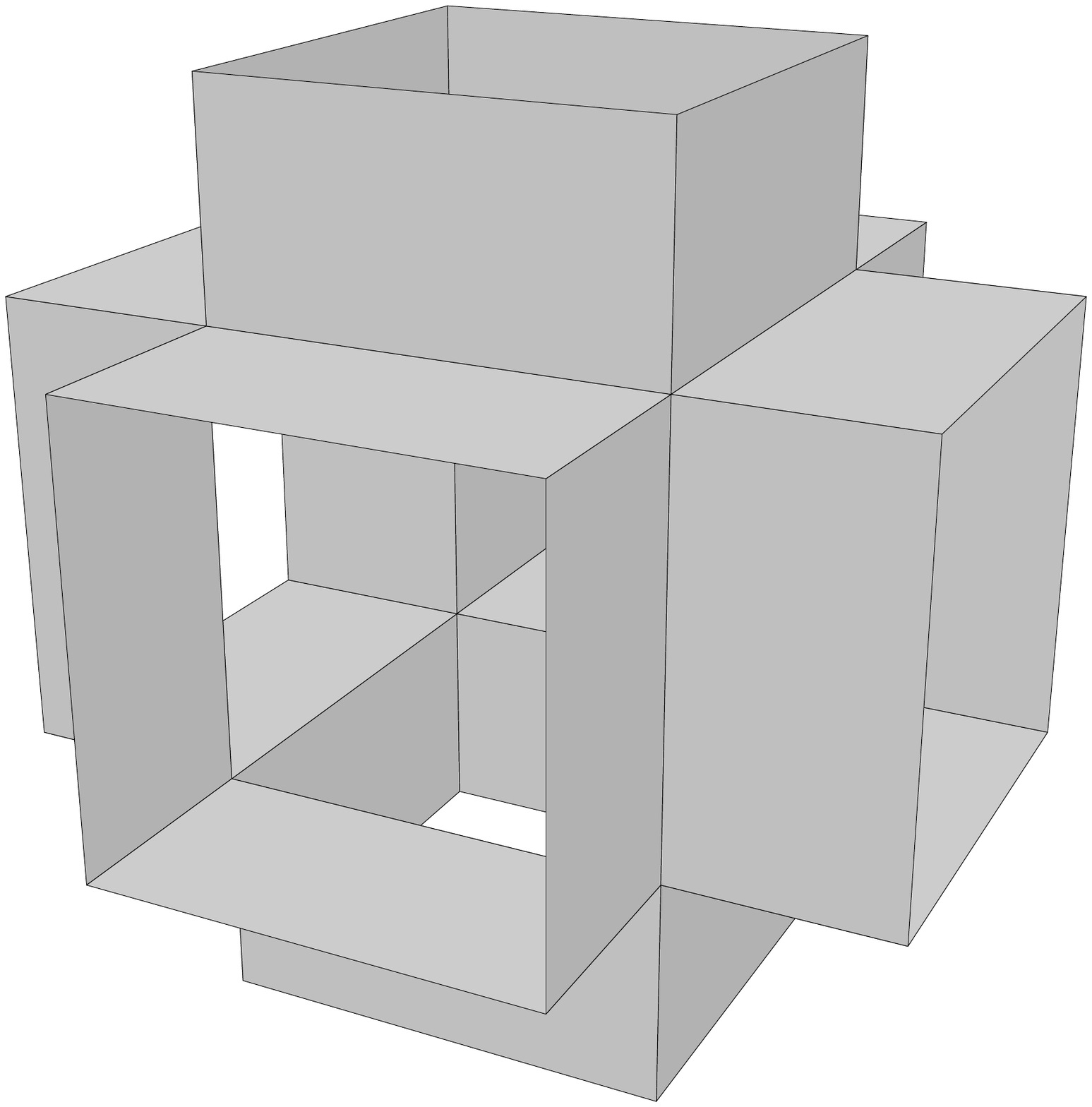}
\vspace{-2cm}
\caption{A fundamental domain of 3-periodic surface}
\label{sur}

\end{figure}

Let us note also that Novikov's problem can be easily translated to the language of measured foliations on a surface: one can consider $\mathbb T^{3}=\pi(\mathbb R^{3})$ and surface $M=\pi(\hat M))$ in the torus. Then, the family of planes indicates a foliation $F$ on the torus (and on $M$), determined by the 1-form $\omega=H_{1}dx_{1}+H_{2}dx_{2}+H_{3}dx_{3}$. 
So, we are interested in the possible behavior of leaves of $F$. In these terms the chaotic case can be described as follows: there exists a component of genus $3$ such that the foliation is minimal on it. 

The same type of foliations was studied in \cite{McM} where it was called a minimal foliation with zero flux. The foliation from \cite{AY} that is used in \cite{McM} as a typical representative of this family is a chaotic regime (and, moreover, belongs to the particular class we deal with, see Figure \ref{sec}). 

\subsection{Systems of isometries}
The notion of systems of isometries was introduced by G. Levitt, D. Gaboriau and F. Paulin in \cite {GLP}.
\begin{definition}
\emph A {system of isometries} $S$ consists of a finite disjoint union $D$ of compact subintervals of the real line $\mathbb R$ (\emph{support multi-interval}) together with a finite number $n$ of partially defined orientation preserving isometries $\phi_{j}: A_{j} \to B_{j}$, where each base of $A_{j}, B_{j}$ is a compact subinterval of $D$. 
\end{definition}
In the current paper $D$ will always be a single interval.

Systems of isometries can be considered as a generalization of IET and interval translation mappings (ITM), so it is natural to define the orbit of such system in the same way as it was done for IET.
\begin{definition}
Two points $x,y$ in $D$ belong to the same $S$-orbit if there exists a word in $\phi_{i}$ and $\phi^{-1}_{i}$ sending $x$ to $y$.
\end{definition}
\noindent We will denote the orbit of the point $x$ by $\Gamma_x$.

Now one can define the equivalence relationship on systems of isometries. 
\begin{definition}
Two systems of isometries $S$ and $S^{'}$ with support intervals $[A,B]$ and $[A^{'},B^{'}]$, respectively, are called \emph{equivalent}, if there is a real number $t \in {\mathbb R}$ and an interval $[A_{0},B_{0}] \subset [A,B] \cap  [A^{'}+t,B^{'}+t]$ such that
\begin {enumerate}
\item every orbit of each of the systems $S$ and $S^{'}+t$ contains a point lying in $[A_{0},B_{0}]$
\item for each point $x\in [A_{0},B_{0}]$ the orbits $\Gamma_{x}(S)$ and $\Gamma_{x}(S^{'})$ as graphs are homotopy equivalent through mappings that are identical on $[A_{0},B_{0}]$ and such that the full preimage of each vertex contains only finitely many vertices of the other graph.
\end {enumerate}
\end{definition}

\noindent One can check that it is an equivalence relation. Informally, the definition means two systems of isometries $S$ and $S'$ are equivalent if there exists $t\in\mathbb R$ such that $S$ and $S'$ indicate the same equivalence relation on $[A,B]\cap[A'+t, B'+t]$. It implies that two equivalent systems have the same behavior of orbits. 

In the current paper we concentrate on a particular class of systems of isometries.
\begin{definition}\label{def1}
A system of isometries $S$ is called \emph{special} if the following restrictions hold:
\begin{itemize}
\item $D$ is an interval of the real line, say, $[0,1]$;
\item $n=3$;
\item all $A_{i}$ start in $0$;
\item all $B_{i}$ end in $1$;
\item $\Sigma_{i=1}^{3}|A_{i}|=1,$ where $|A|$ means the length of the subinterval $A$;
\item $|A_{1}|>|A_{2}|>|A_{3}|.$
\end{itemize}
\end{definition}

See, for example, Figure \ref{S}.
\begin{figure}[h]
\centering
\vspace{-4cm}
\includegraphics[height=10cm, width=8cm]{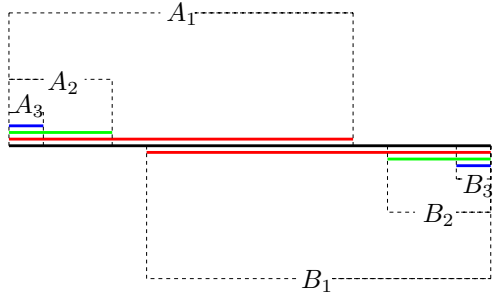}
\put(-142,173){$A_{1}$}
\put(-91,72){$B_{1}$}
\put(-187,147){$A_{2}$}
\put(-45,97){$B_{2}$}
\put(-200,138){$A_{3}$}
\put(-30,108){$B_{3}$}
\vspace{-2.5cm}
\caption{System of isometries}
\label{S}
\end{figure}

So, any special system $S$ can be described in the following way:
\begin{equation}
\begin{split}
\label{100}
S=([0,a+b+c];& [0,a]\leftrightarrow[b+c,a+b+c],\\
                    & [0,b]\leftrightarrow[a+c,a+b+c], \\
                    & [0,c]\leftrightarrow[a+b,a+b+c])
\end{split}
\end{equation}
with $a>b>c>0$, $a+b+c=1$.

We are only interested in the most generic case of special
system of isometries in the sense that no integral linear relation holds for the parameters $a, b, c$ except those that must hold by definition.
 
We concentrate on special systems of isometries of \emph{thin type}. By the latter we mean a system of isometries for which an equivalent system may have arbitrarily small support (or, equivalently, all orbits are everywhere dense). Thin type was discovered by Levitt in \cite{L} and sometimes is mentioned as \emph{Levitt} or \emph{exotic} case. The term ``thin" was introduced by M. Bestvina and M. Feign (see \cite{BF} for a formal definition of band complex of thin type in terms of the Rips machine). 

For the future construction we need the following obvious 
\begin{proposition}\label{thin}
A special system of isometries $S$ is not of thin type if $a<b+c$.
\end{proposition}
See Subsection \ref{RI} for a formal proof.

As it was mentioned above, Dynnikov in \cite{D} showed a strategy how to construct a symmetric 3-periodic surface whose intersections with a family of planes of fixed direction have chaotic behavior using a system of isometry of thin type (see subsection \ref{DC} for details).

\subsection{Suspension complex}
Here we recall briefly the construction of the suspension complex for systems of isometries from \cite{GLP}. It can be considered as an analogue of zippered rectangles model suggested by W. Veech (\cite{V}). 

With each special system of isometries we can associate a foliated 2-complex $\Sigma$ (in terms of $\mathbb R$-trees theory, it is a \emph{band complex}). Start with the disjoint union of the support interval (foliated by points) and strips $A_{j}\times [0, 1]$ (foliated by ${*} \times [0, 1]$). We get $\Sigma$ by glueing $A_{j}\times [0, 1]$ to $D$, identifying each $(t,0) \in A_{j}\times {0}$ with $(t,0) \in A_{j} \subset D$ and each $(t,1) \in A_{j}$  with $\phi_{j}(t) \in B_{j}\subset D$. We will identify $D$ with its image in $\Sigma$.
Thus, one gets a 2-dimensional complex with a vertical foliation on it.

Our family of band complexes is a particular class of what appears in geometric group theory as an instrument for describing actions of free groups on $\mathbb R$-trees (see, for example, \cite{BF} for details).

\subsection{Dynnikov's model}\label{DC}
In the current section we explain how to construct a chaotic regime if one has a special system of isometries $S$ of thin type. The construction for some more generic case was introduced by Dynnikov in \cite{D} (see also \cite{S}). 

Indeed, we consider a piecewise smooth surface in the 3-torus $\mathbb T^{3}$ that will be described below, and study the asymptotic behavior
of sections of $\mathbb Z^{3}$-covering of this surface in $\mathbb R^{3}$ by a family of parallel planes $\alpha : H_{1}x_{1} + H_{2}x_{2}+H_{3}x_{3} = const$, where $H$ is some fixed covector.
For technical reasons, we will vary not the covector $H$ but the coordinate system and the fundamental domain of the lattice in $\mathbb R^{3}$ so as to have the coordinates of $H$ constant and
equal to $(0, 1, 0)$.  

The main idea of the construction is as follows. There is a band complex $\Sigma$ associated with the system $S$. Using the parameters of the system $S$ one can define a lattice $\Lambda$ and a fundamental domain of the surface $M$ that provide a triply periodic surface $\hat M$ which is an abelian cover of $M$ with respect to the lattice $\Lambda$. We do it in a way that the sections of $\hat M$ by the fixed family of planes have the sections of the abelian cover of $\Sigma$ as a deformation retract. 

In order to construct $M$ one has to replace the support interval of the complex $\Sigma$ by a cylinder and make three cuts that are posed in accordance with the position of bands; then the boundary of these cuts that correspond to the same band should be glued to each other (or, equivalently, glued up by cylinders). Then the saddles of the foliation appear on the ends of cuts. The same idea was used in \cite{DS} (see Figure 7 there for the illustration). 
We proceed with the detailed description.

\begin{figure}[h]
\vspace{-2.5cm}
\includegraphics[width=7cm,height=10cm]{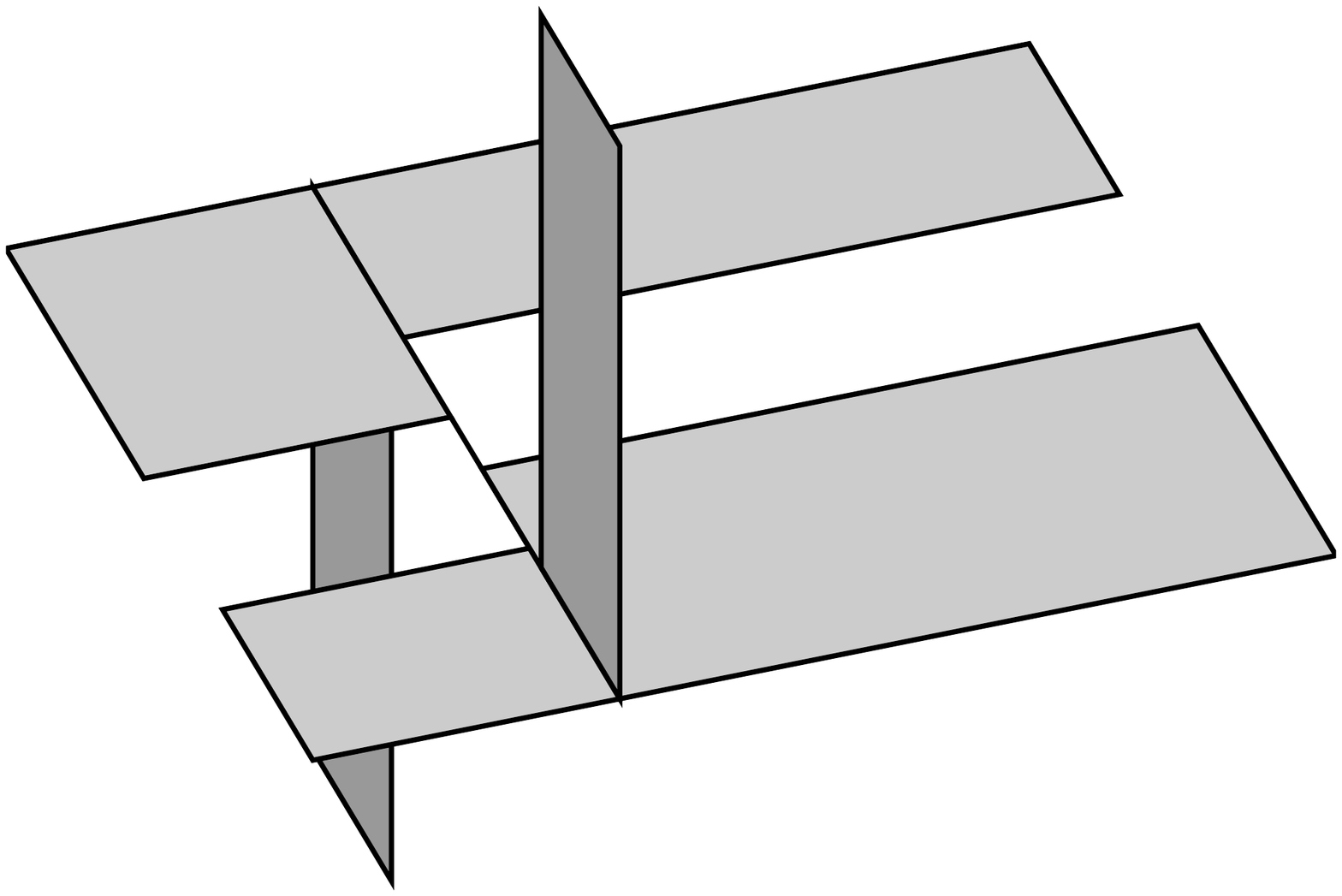}
\put(-113,135){\vector(0,1){15}}
\put(-100,120){\vector(4,1){15}}
\put(-120,120){\vector(-4,-1){15}}
\put(-120,185){\vector(-4,-1){15}}
\put(-170,175){\vector(4,1){15}}
\put(-145,132){\vector(0,1){15}}
\put(-107,195){\vector(0,1){28}}
\put(-9,134.6){\vector(4,1){25}}
\put(-147,182){\vector(-1,2){7}}
\put(-105,210){Oz}
\put(2,130){Ox}
\put(-154,200){Oy}
\vspace{-3cm}
\caption{Suspension complex}
\label{C}
\end{figure}

\begin{figure}[h]
\centering
\vspace{-24cm}
\hspace{-5cm}
\includegraphics[scale=1.1]{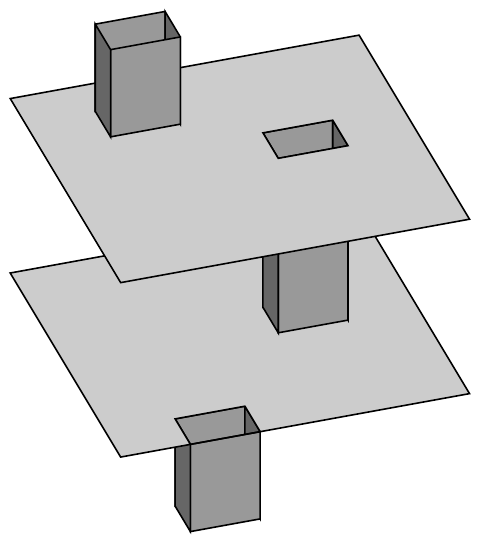}
\vspace{-3cm}
\caption{The surface $M_{0}$}
\label{f}
\end{figure}

We take $S$ identified by $a,b,c$ as was described above such that $S$ is of thin type. Let us introduce the following notation for the rectangles in the
plane $\mathbb R^{2}$:
$$T_{1} = [0, 1]\times [0, a + b + 2c];$$
$$T_{2} = [1/5, 2/5]\times [0, c];$$
$$T_{3} = [3/5, 4/5]\times [a, a + c];$$
$$T_{4} = [1/5, 2/5]\times [a + b , a + b + c].$$
One can easily check that $T_{2}, T_{3}, T_{4}\subseteq [0,1]\times [0,1].$
As a fundamental domain $M_{0}$ (see Figure \ref{f}) of the surface $\hat M$, we take the following piecewise linear surface:
$$(T_{1} \setminus (T_{2}\cup T_{3}))\times\frac{1}{4} \cup (T_{1} \setminus (T_{3}\cup T_{4})) \times \frac{3}{4} \cup \partial T_{2} \times [0, \frac{1}{4}] \cup\partial T_{3}\times[\frac{1}{4}, \frac{3}{4}] \cup \partial T_{4} \times [\frac{3}{4}, 1].$$

The lattice $\Lambda$ is spanned by the following three vectors:
$$e_{1} = (1, -b-c, 0), e_{2} = (1, a + c, 0), e_{3} = (0, a + b,1).$$
The covering surface $\hat M$ is equal to $M_{0} + G$, where G is a translation group based on $\Lambda$ (see Figure \ref{f2}).

\begin{figure}
\centering
\vspace{-10cm}
\hspace{-2.5cm}
\includegraphics[scale=0.7]{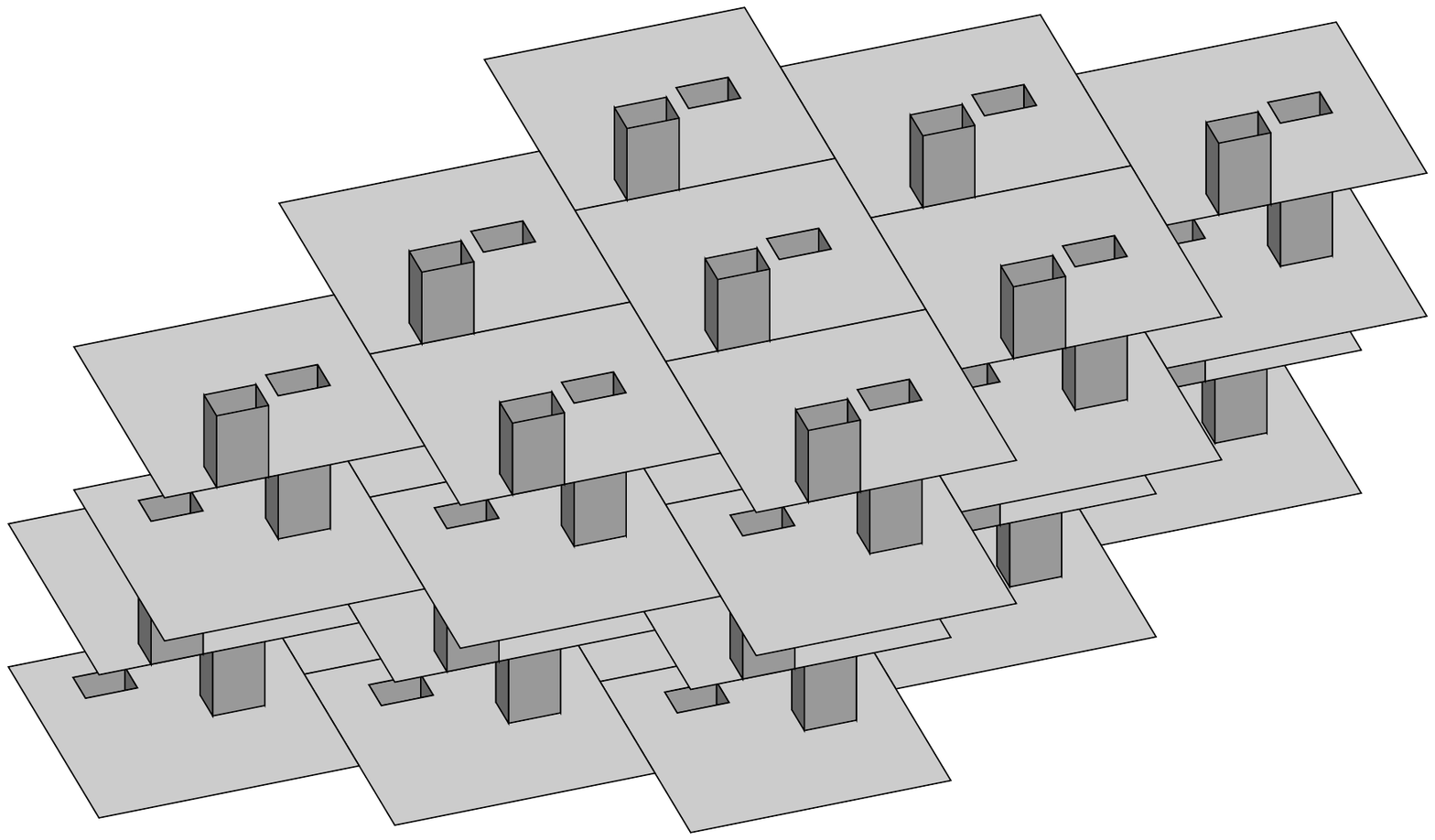}
\vspace{-1cm}
\caption{The surface $\hat M$}
\label{f2}
\end{figure}
As mentioned above, $H=(0,1,0)$.

\begin{theorem}\label{d08}[Dynnikov, 2008]
Let  $S$ be a special system of isometries of thin type identified by the parameters $(a, b, c)$. Then, the sections of the surface $\hat M$, constructed as above with these values
of the parameters, by any plane orthogonal to H = (0, 1, 0) are chaotic.
\end{theorem}
\begin{remark}
See Figure \ref{sec} for the example of the chaotic plane section.
\end{remark}
\begin{proof}
Our proof repeats the proof of the similar theorem in \cite{D}.
Let us denote by M the image of the projection of $\hat M$ in the torus $\mathbb T^{3}=\mathbb R^{3}/\Lambda$: $M = \pi(\hat M)$. For
studying the sections $\alpha \cap \hat M$, we consider the foliation $F$ on $M$ defined by a restriction to $M$ of the 1-form $dx_{2}$. It is easy to see that
the surface $M$ has genus 3. We need to show that the foliation $F$ is minimal that is the closure of any leaf of $F$ coincides with $M$. In other words, one needs to check that $F$ does not have closed leaves or saddle connection cycles. 

It can be directly checked that the two saddles we work with belong to different planes of the form $x_{2} = const$; hence, saddle connections between two different saddles do not exist.

In order to show that $F$ does not have closed leaves and separatrix cycles, we consider not the surface $M$ itself but one of the two parts into which it cuts the torus $\mathbb T^{3}$. Both parts are filled handlebodies of
genus $3$. We denote one of them (which contains a point $\pi(0, 0, 1/2)$) by $M_{1}$ and the $\mathbb Z^{3}$-covering
of $M_{1}$ by $\hat M_{1}$. Now, we consider some natural modification of the suspension complex that we constructed in the previous section. $\Sigma^{'}$  is a 2-dimensional complex in $\mathbb T^{3}$ consisting of 3 rectangles:
$$P_{1} = \pi([3/10, 13/10]\times [0,a]\times{1/2})$$
$$P_{2} = \pi([3/10, 13/10]\times [a+c,a+b+c]\times{1/2})$$
$$P_{3} = \pi({3/10}\times[0,c]\times [1/2, 3/2])$$

Figure \ref{C} depicts a particular model of $\Sigma^{'}$ viewed as subset of the torus. 

The only difference between $\Sigma$ and $\Sigma'$ is that we glued up both of horizontal sides of rectangles to the support interval (in case of $\Sigma$) while for $\Sigma^{'}$ we just leave them free. Since we work with a system of isometries of thin type, due to Proposition \ref{thin} this difference disappears when we consider an abelian cover of $\Sigma^{'}$ (let us denote it by $\hat \Sigma^{'}$, see Figure \ref{complex}). One can directly check that the following holds:
\begin{lemma}\label{dynlem}
For any plane $\alpha$ defined by an equation of the form
$x_{2} = const$, the section $\hat M\cap\alpha$ has $\hat\Sigma^{'}\cap\alpha$ as a deformation retract and the restriction of the form $\omega = dx_{2}$ to $\Sigma^{'}$ defines a vertical foliation on the band complex (see figure \ref{DR} where sections $\hat M\cap\alpha$ are given by gray rectangles and sections $\hat\Sigma^{'}\cap\alpha$ are represented by solid black lines).
Moreover, the deformation is finite and uniformly bounded.
\end{lemma}

\begin{figure}
\centering
\vspace{-10cm}
\hspace{-2.5cm}
\includegraphics[scale=0.7]{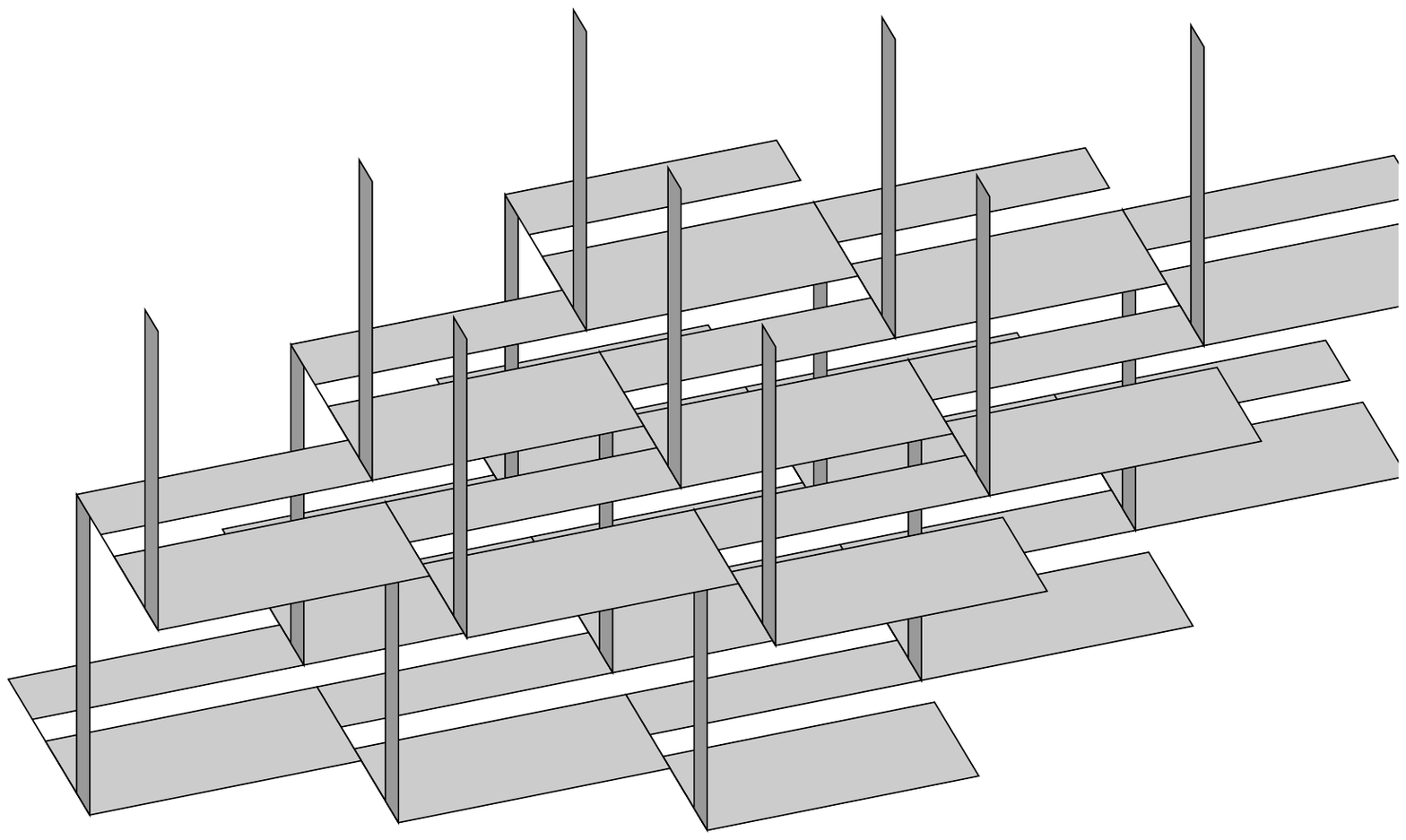}
\vspace{-3cm}
\put(-160,88){$\tilde a$}
\put(-230,91){$\tilde a$}
\put(-310,94){$\tilde a$}
\put(-230,110){$\tilde b$}
\put(-310,110){$\tilde b$}
\put(-160,108){$\tilde b$}
\put(-342,120){$\tilde c$}
\put(-338,122){\vector(1,0){10}}
\put(-327,100){\vector(0,1){30}}
\put(-310,90){\vector(4,1){10}}
\put(-325,170){$\tilde c$}
\put(-321,172){\vector(1,0){10}}
\put(-290,120){\vector(-4,-1){10}}
\caption{Complex $\hat\Sigma^{'}$}
\label{complex}
\end{figure}

\begin{figure}
\centering
\vspace{-1cm}
\includegraphics[scale=0.4]{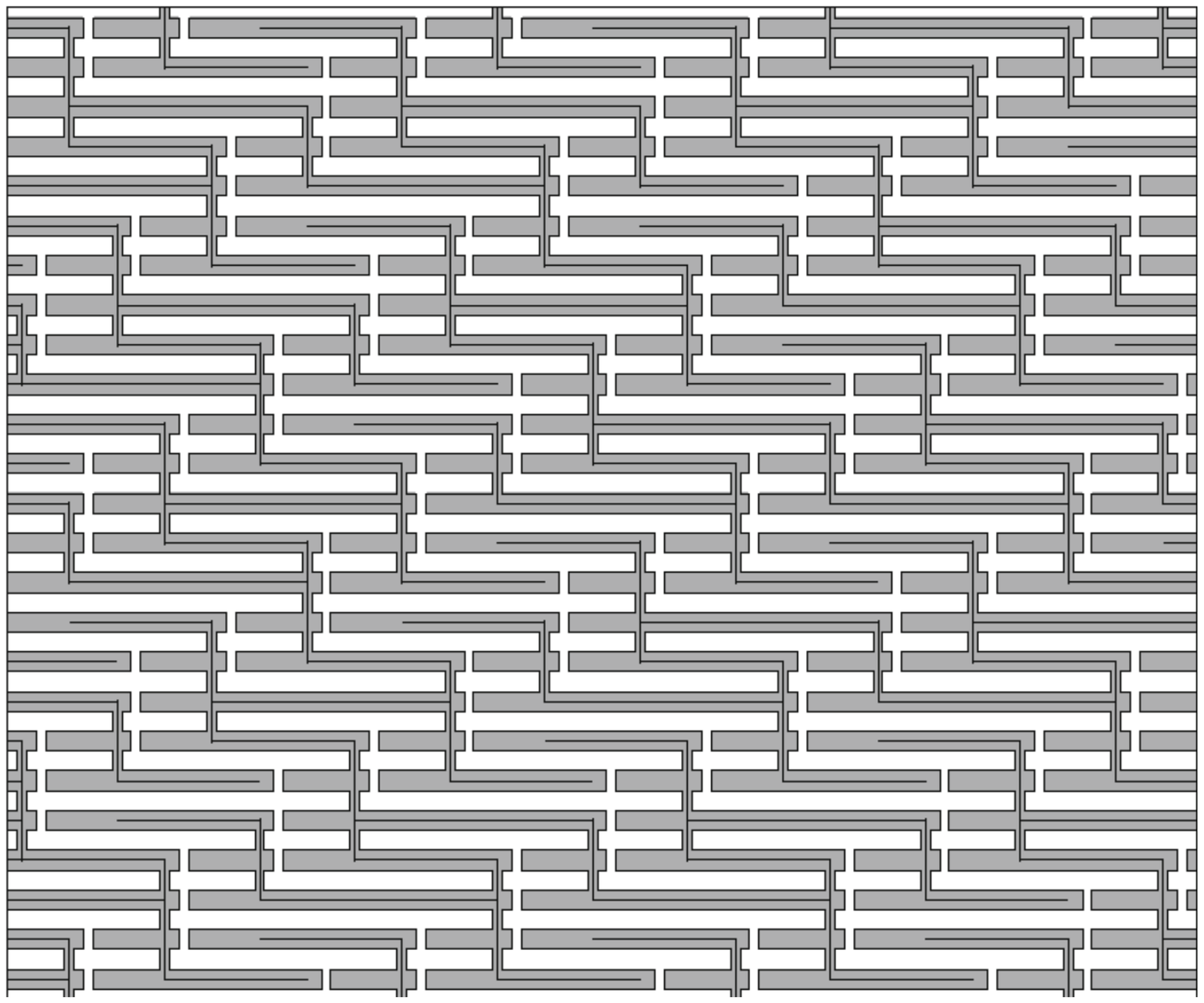}
\vspace{-6cm}
\caption{The section $\hat M\cap\alpha$ has $\hat\Sigma^{'}\cap\alpha$ as a deformation retract}
\label{DR}
\end{figure}

The foliation $F$ has closed leaves or separatrix cycles if and only if the sections of the manifold with boundary
 $\hat M_{1}$ by planes  $x_{2} = const$ have either compact or non-simply-connected regular components.
Hence, the same has to be true after replacing  $\hat M_{1}$ by $ \hat\Sigma^{'}$. Furthermore, the latter can be
reformulated in terms of the system $S$ by saying that it must have an essential set of finite
orbits and an essential set of non-simply connected ones, but we know that $S$ is minimal since $S$ is of thin type.
 
\end{proof}
\section{Symbolic dynamics}\label{2}
The main result of this section is that on a space of special systems of isometries one can define an analogue of the symbolic dynamics on a space of IET (see \cite{BG}). This dynamics is described by a Markov shift that satisfies the Big Images and Preimages Property.

\subsection{The Rauzy induction}\label{RI}
In the theory of IETs the Rauzy induction is a Euclid type algorithm that transforms an original IET into another one operating on a smaller interval but equivalent from the point of view of topology of the corresponding measured foliation. Its iteration can be viewed as a generalized version of continued fraction expansion. This process can also be considered as a variation of the Rips machine algorithm for band complexes in the theory of
{$\mathbb R$}-trees (\cite{GLP}). 

We study modification of this algorithm for our purpose. The main idea is that from any system of isometries one constructs a sequence of systems of isometries equivalent  to the original one but with a smaller support. Combinatorial properties of this sequence are responsible for ``ergodic'' properties of the original system of isometries. The Rauzy induction for systems of isometries  was introduced by Dynnikov in \cite{D}.

\emph{The Rauzy induction} for a special system of isometries is a recursive application of admissible transmissions followed by reductions as described below.

\begin{definition}
Let $$S=([0,a+b+c];[0,a]\leftrightarrow [b+c,a+b+c];$$
$$ [0,b] \leftrightarrow [a+c,a+b+c]; [0,c]\leftrightarrow [a+b, a+b+c])$$
be a special system of isometries. So, two of the subintervals,
$\left[a+c,a+b+c\right]$ and $\left[a+b,a+b+c\right]$, say, are contained in the third one $\left[b+c,a+b+c\right]$, say. Let $S^{'}$ be the system of isometries obtained from $S$ by replacing the pair $\left[0,b\right]\leftrightarrow\left[a+c,a+b+c\right]$
by the pair $\left[0,b\right]\leftrightarrow\left[a-b,a\right]$ and the pair $\left[0,c\right]\leftrightarrow\left[a+b,a+b+c\right]$ by the pair $\left[0,c\right]\leftrightarrow\left[a-c,a\right]$

We say that $S^{'}$is obtained from $S$ by a \emph{transmission} (on the right).
\end {definition}

\begin{definition}
Let
$$
S=\left(\left[A,B\right];\left[a_{1},b_{1}\right]\leftrightarrow\left[c_{1},d_{1}\right];\left[a_{2},b_{2}\right]\leftrightarrow\left[c_{2},d_{2}\right];\left[a_{3},b_{3}\right]\leftrightarrow\left[c_{3},d_{3}\right]\right)
$$
be a system of isometries (not necessarily special) and let $d_{1}=B$. We call all
endpoints of our subintervals \emph{critical points}. Assume that
the point $B$ is not covered by any interval from $S$ except $[c_1,d_{1}]$
and that the interior of the interval $\left[c_{1},d_{1}\right]$ contains
a critical point. Let $u$ the rightmost such point. Then the interval $[u,B]$ is covered only by one interval from our system. Replacing the pair $\left[a_{1},b_{1}\right]\leftrightarrow\left[c_{1},d_{1}\right]$
with $\left[a_{1},b_{1}-d_{1}+u\right]\leftrightarrow$$\left[c_{1},u\right]$
in $S$ with simultaneous cutting off the part $[u,B]$ from the support interval will be called a\emph{ reduction on the right} (of the pair $\left[a_{1},b_{1}\right]\leftrightarrow\left[c_{1},d_{1}\right]$). 
\end{definition}
Note that the application of the Rauzy induction to a special system of isometries gives us a special system of isometries again (see Figure \ref{R}). The pair of subintervals that was reduced is called a \emph{winner} (like in the case of IET).

\begin{figure}[h]
\includegraphics[width=9cm,height=11cm]{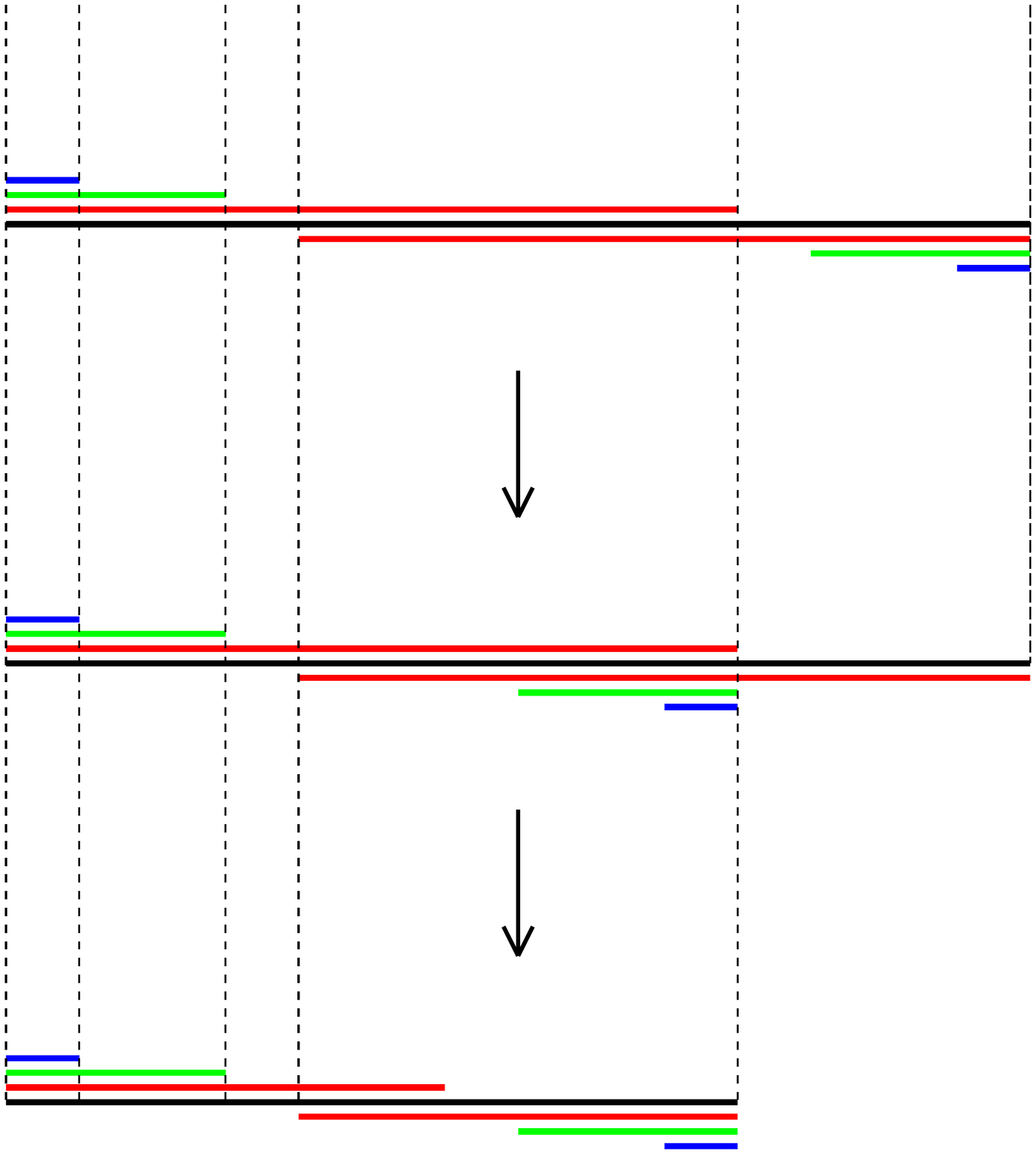}
\put(-214,265){$c$}
\put(-230,265){$0$}
\put(-185,265){$b$}
\put(-178,265){$b+c$}
\put(-83,265){$a$}
\put(-40,265){$a+b+c$}
\vspace{-1cm}
\caption{The Rauzy induction: transmission of $b$ and $c$ intervals and reduction of $a$-interval.}
\label{R}
\end{figure} 

We have the following obvious 
\begin{lemma}
The Rauzy induction does not influence the existence of the finite orbits or their property to be everywhere dense: the origin and the image are equivalent.
\end{lemma} 

We say that a system of isometries has a \emph{hole}
if there are some points in the support interval that are not covered by
any interval from $S$. This means in particular that our system has points
with finite orbits. Therefore, one can stop the Rauzy induction once it results in a system with a hole.

For instance, one can check directly that the hole appears after one step of the Rauzy induction if $a<b+c$.

One can check that a system of isometries of thin type is exactly such a system for which the Rauzy induction can be applied for infinite number of times, and we never get a hole (see \cite{S12} for details).
It is easy to compare the formulas for the Rauzy induction with the maps that appear in \cite{AS} in description of the Rauzy gasket as iterated function system. So, one can check that the set of lengths of subintervals $(l_1,l_2,l_3)$ with the renormalization condition $\sum_{i=1}^{3}l_i=1$ such that the corresponding special systems of isometries are of thin type, forms the Rauzy gasket.
 
Let $\EuScript{A} =\{ 1,2,3 \}$; let also $l=(l_1,l_2,l_3)$ be the vector of lengths of the subintervals and $(a,b,c)$ be a permutation of $l_i$ in a decreasing order: 
$$a=\max_{i=1,2,3}l_i, c=\min_{i=1,2,3}l_i$$ and $b$ will be equal to the medium value $l_i$. For sake of simplicity we can assume that for the original system the intervals were enumerated in decreasing order: $a=l_1, b=l_2, c=l_3$. 

Now let us check what happens with the system of isometries under the action of the Rauzy induction. 
Application of the Rauzy induction can change the order of the intervals; in order to keep track of it, as in case of IET, we add some combinatorial data. Namely, for each special system of isometries we associate not only the collection of three lengths ($l_1,l_2,l_3$) but also a permutation $\tau$ such that $\tau(l_1,l_2,l_3)$ is in decreasing order: for instance, $\tau=(1,2,3)$ if $l_1>l_2>l_3$; similarly, $\tau=(2,3,1)$ if $l_2>l_3>l_1$ etc.

Thus, the parameter space $\EuScript{V}= \mathbb R^{3}\times S_{3},$ with the normalizing restriction $a+b+c=1.$ 

Assuming that we started from $(1,2,3)$, one step of the Rauzy induction can be described as follows:
\begin{itemize}
\item if the order is preserved and so $\tau=(1,2,3)$: 
$$ \left( \begin{array}{c} a \\ b \\ c \end{array} \right) = \begin{pmatrix} 1 & 1 & 1 \\ 0 & 1 & 0  \\ 0 & 0 & 1 \end{pmatrix} \cdot \left(\begin{array}{c} a' \\ b'\\c' \end{array} \right).$$
We denote the matrix of the induction by 
$$R_{1}=\begin{pmatrix}
1 & 1 & 1\\
0 & 1 & 0 \\
0 & 0 & 1
\end{pmatrix};$$
 
\item if the order changes in the following way: $\tau=(2,1,3)$; $$ \left( \begin{array}{c} a \\ b \\ c \end{array} \right) =R_{2} \cdot \left(\begin{array}{c} a' \\ b'\\c' \end{array} \right),$$
where  $$R_{2}=\begin{pmatrix}
1 & 1 & 1\\
1 & 0 & 0 \\
0 & 0 & 1
\end{pmatrix};$$

\item if the new order of the intervals is given by  $\tau=(3,1,2)$, the induction is described as follows: $$ \left( \begin{array}{c} a \\ b \\ c \end{array} \right) =R_{3} \cdot \left(\begin{array}{c} a' \\ b'\\c' \end{array} \right),$$ where
 $$R_{3}=\begin{pmatrix}
1 & 1 & 1\\
1 & 0 & 0 \\
0 & 1 & 0
\end{pmatrix}.$$

\end{itemize}
\subsection{The acceleration}
\noindent One can construct an \emph{accelerated} version of the Rauzy induction. We define a \emph{generalized iteration} of the Rauzy induction by analogy with a step of the fast version of Euclid's algorithm, which involves the division with remainder instead of substraction of the smaller number from the larger. It may happen that only one of the three pairs of intervals is subject to reduction in several consecutive steps of the Rauzy induction (and the intervals from the second and the third pair are involved only in transmissions). Another words, it means that there is the same winner for several consecutive steps of the algorithm.
In this case we consider the result of such a sequence of the Rauzy induction iterations as the result of applying of one generalized iteration. This kind of acceleration for IET was described by Zorich in \cite{Z}. 

\noindent  The matrix of the one step of the accelerated Rauzy induction $R(n)$ is the following:
$$\begin{pmatrix}
n & 1 & n\\
1 & 0 & 0 \\
0 & 0 & 1
\end{pmatrix}$$
or 
$$\begin{pmatrix}
n & n & 1\\
0 & 0 & 1 \\
0 & 1 & 0
\end{pmatrix},$$
where $n$ is a number of simple Rauzy inductions included in one generalized iteration.
There is an evident 
\begin{lemma}
The matrix of a single step of the accelerated Rauzy induction is given by the following formulas:
$$R(n)=R_1^n\cdot R_2$$
or
$$R(n)=R_1^n\cdot R_3,$$
where $n$ is a number of simple Rauzy iterarions included in one generalized iteration. 
\end{lemma}

\subsection{The Markov Map}
In the case of special systems of isometries $X$ - parameter space - is the triangle with vertices $(1:0:0)$, $(0:1:0)$, $(0:0:1)$. 
The Rauzy induction defines a partition of $X$ in the following way:
\begin{itemize}
\item on step zero $X$ is divided into four subsimplices: 
\begin{itemize}
\item $X^{0}_{1}$ with vertices $(1:1:0)$, $(1:0:0)$, $(1:0:1)$. This subsimplex can be described by the following inequality: 
$a>b+c$ and so includes two regions: $l_1>l_2+l_3, l_2>l_3>0$ and $l_1>l_2+l_3$, $l_3>l_2>0$; the first one corresponds to the coding $(1,2,3)$; the second one is coded by $(1,3,2)$;
\item $X^{0}_{2}$ with vertices $(0:1:1)$, $(0:1:0)$, $(1:1:0)$, it corresponds to the coding $(2,1,3)$ and $(2,3,1)$;
\item $X^{0}_{3}$ with vertices $(1:0:1)$, $(0:0:1)$, $(0:1:1)$, it corresponds to the coding $(3,1,2)$ and $(3,2,1)$;
\item $X^{0}_{0}$ with vertices $(1:0:1)$, $(0:1:1)$, $(1:1:0)$, it corresponds to the hole; 
\end{itemize}
\item the renormalized version of the induction map is given by the following formula 
$$T: X\to X, T(l)=\frac{Rl}{||Rl||},$$
where $R$ is the matrix of the induction and $l=(a,b,c)$ with $a+b+c=1$.
\item after one step of the Rauzy induction one of three subsimplices (depending where the point that we examine was located) will be also divided into four parts in the same way etc.
\end{itemize}
We enumerate steps of the (non-accelerated) induction by lower $n$ and the number of the part of each step by the upper index $i: X^{i}_{n}$ is the cell with the corresponding address.
\begin{lemma}
$T$ is a Markov map, and $(X^{i}_{n})$ is a Markov partition.
\end{lemma}

The Markov partition is shown on Figure \ref{RG}; the Rauzy gasket (black part) is a fractal subset of $X$ determined by the systems of isometries of thin type; the white part corresponds to the systems of isometries such that the hole was obtained after some steps of the Rauzy induction. 

\begin{figure}[h]
\centering
\vspace{-20cm}
\hspace{-5cm}
\includegraphics[scale=1.1]{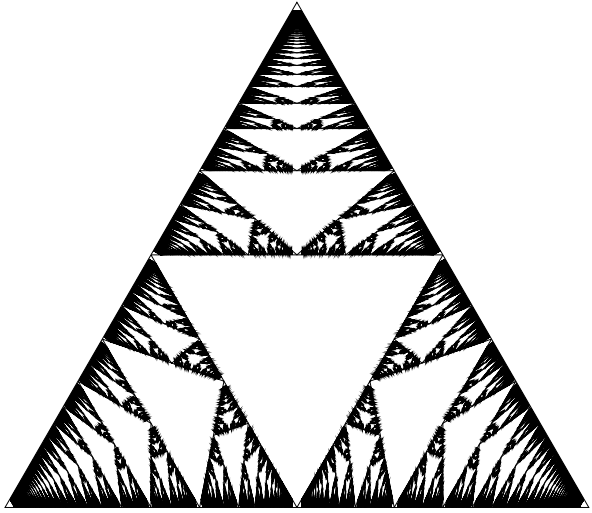}
\vspace{-4cm}
\caption{The Rauzy gasket}
\label{RG}
\end{figure}

\subsection {The Rauzy graph}
As it was mentioned above, each special system of isometries is described by the vector of lengths $(a,b,c)$ and the permutation $\tau$.
So, like in case of IET, we use \emph{the Rauzy graph} to describe the combinatorics of the accelerated Rauzy induction. 
 
Then, vertices of the Rauzy graph are all permutations of 3 elements, and 2 vertices are connected by an arrow if and only if there exists a realization of it by the Rauzy induction. For example, looking at one step of the Rauzy induction, we see that there is $(1,2,3)\rightarrow (2,1,3)$ but there is no arrow between $(1,2,3)$ and $(3,2,1).$ However, it could be possible that the induction stops after a while because we obtained a hole. In the current paper we work only with the systems of isometries of thin type, and the hole vertex can be excluded from the graph (we call this exclusion an ``adjustment''). 

So it is enough to consider the graph on $6$ vertices defined by the described permutations.

Acceleration means that the combinatorics changes after each step (the graph does not contain loops). 
The adjusted Rauzy graph for the accelerated Rauzy induction is shown on Figure \ref{G}.

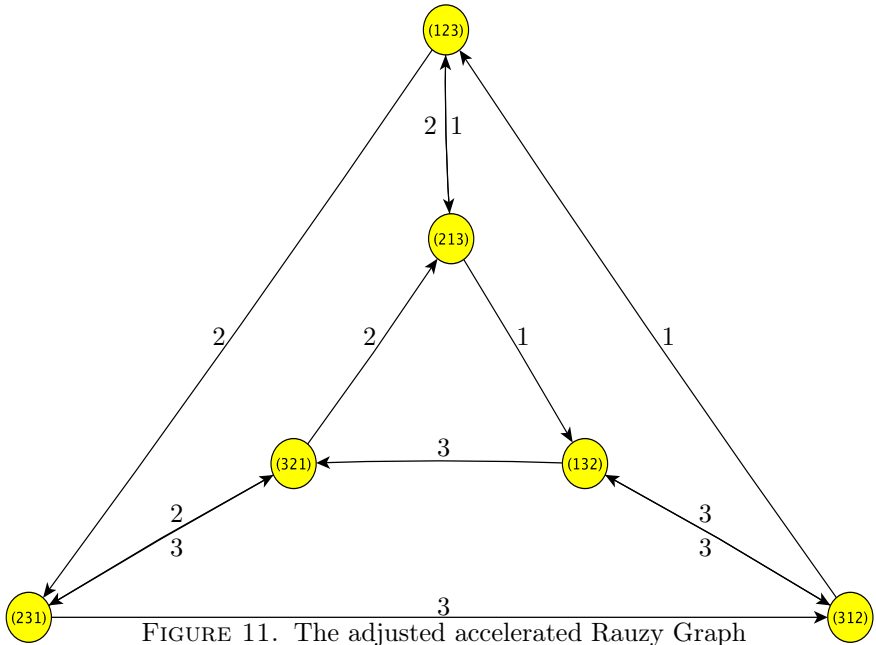
\begin{figure}[h]
\includegraphics[width=12cm,height=9cm]{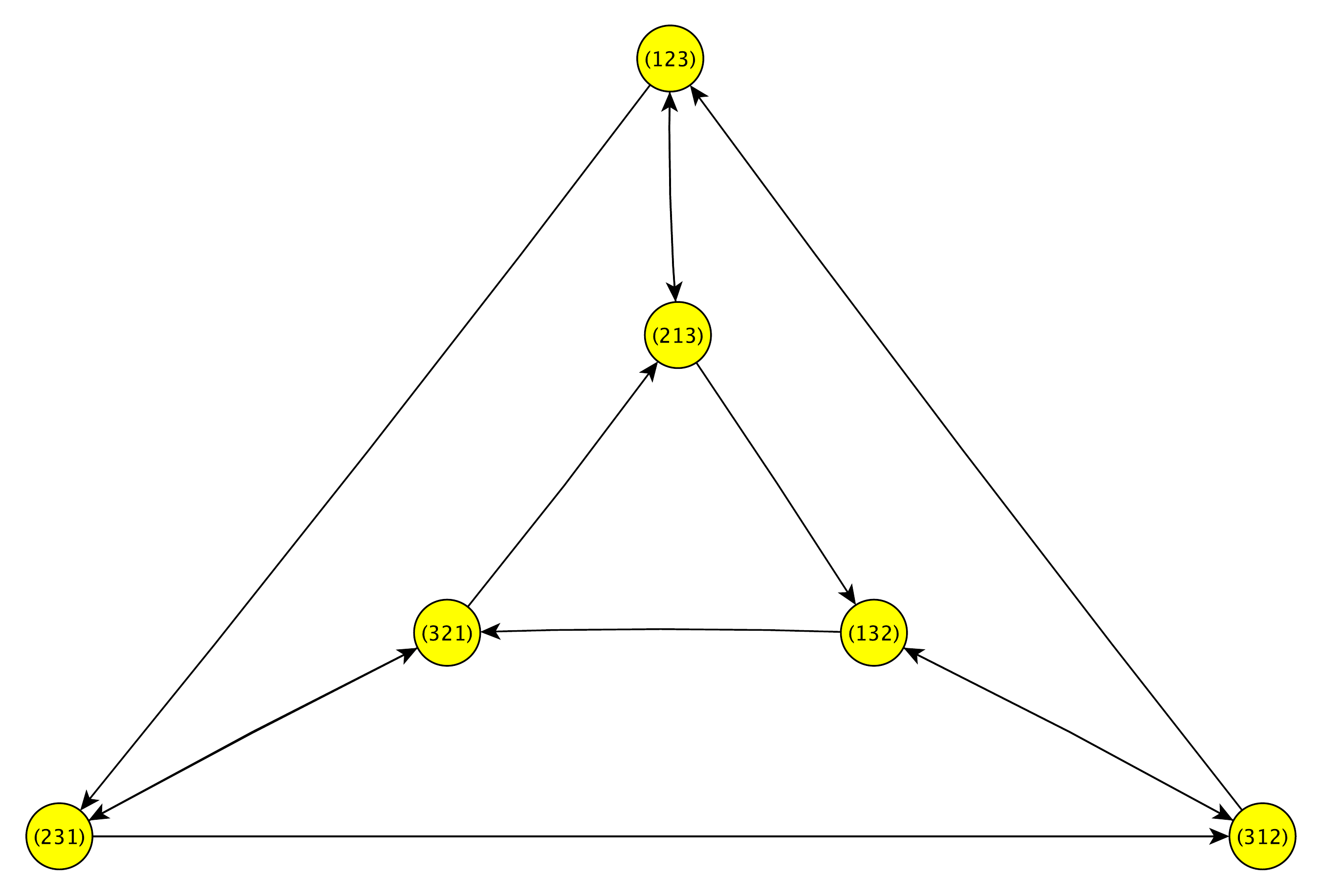}
\vspace{-1cm}
\put(-145,120){$1$}
\put(-90,120){$1$}
\put(-170,200){$1$}
\put(-260,120){$2$}
\put(-203,120){$2$}
\put(-180,200){$2$}
\put(-175,78){$3$}
\put(-175,18){$3$}
\put(-276,40){$3$}
\put(-276,53){$2$}
\put(-76,40){$3$}
\put(-76,53){$3$}
\caption{The adjusted accelerated Rauzy Graph}
\label{G}
\end{figure} 

We have the following obvious 
\begin{lemma}
The [adjusted] Rauzy graph is the Cayley graph for the permutation group $S_3$.
\end{lemma}
It implies in particular that the Rauzy graph is connected.

For future constructions we will also need one additional definition: 
\begin{definition}
A path $\gamma$ in the Rauzy graph is called \emph{complete} if every $\alpha \in \EuScript A$ is a winner of some arrow in $\gamma$.
\end{definition}
\noindent The winners are labeled on edges of the Rauzy graph (see Figure \ref{G}).

\subsection {The Markov shift}
One can also consider the action of the non-accelerated Rauzy induction on the accelerated adjusted Rauzy graph. Then, each vertex of the adjusted Rauzy graph will be decomposed into countable number of vertices, and the same happens to the corresponding Markov cell. Each small Markov cell is coded by a permutation (that comes from the coding of vertices of the accelerated Rauzy graph) and a natural number $n$ of steps of the ordinary induction in the corresponding step of the accelerated one.  Then the Rauzy induction provides the Markov shift $\sigma$ in this coding on a countable alphabet. One can associate in a natural way a graph $\Gamma$ with such a Markov shift. $\Gamma$ can be obtained from the Rauzy graph by dividing every vertex into a countable number of vertices and adding required arrows between these new vertices. 

\begin{definition}
A countable Markov shift $\sigma$ with transition matrix $U$ and set of states $\EuScript{S}$ satisfies \emph{the big images and pre-images property}(BIP) if there exist $\{i_{1},\cdot,i_{m}\} \in \EuScript{S}$ such that for all $j \in \EuScript{S}$ there are $1\leq k,l \leq m$ for which $u_{i_{k},j}u_{j,i_{l}}=1.$
\end{definition} 

\begin{definition}
A Markov shift is \emph{topologically mixing } if for any $i,j\in \EuScript{S}$ there exists a number $N=N(i,j)$ such that for any $n\geq N$ there is an admissible path of length $n$ on the graph of the shift that connects $i$ and $j$.
\end{definition}
\begin{lemma}
The Rauzy induction defines a countable topologically mixing Markov shift that satisfies BIP property.
\end{lemma}
\begin{proof}
The first part follows from the fact that both of the graphs of the induction are connected. 
In order to obtain BIP property we have to choose $m=6$ and $i_{j}$ each belong to a different vertex of the accelerated Rauzy graph. 
\end{proof}

In a standard way for a Markov shift we associate with any finite path $\gamma$ in $\Gamma$  passing through the vertices $j_{1},\cdots, j_{k}$ a \emph{word} $w=(j_{1}\cdots j_{k}).$ 

Let us denote the space of infinite paths in the graph $\Gamma$ by $W$.
Then, to every word $w$ we assign the cylinder of depth $k:$ $C(w) = \{x\in W: (x_{1},\cdots,x_{k})=w\}.$ 

\subsection{The standard cocycle}
One can apply the Rauzy induction not only to a system of isometry but also to a corresponding suspension complex. 
A suspension complex for a special system of isometry contains three bands, each of which has a width (horizontal length) and a length (more precisely, vertical length). The matrix $R$ described above tells us how the widths of bands are cut by the Rauzy induction. 
At the same time, the vertical lengths of the bands increase during the same procedure (see \cite{BF} for the description of the Rips machine application to the band complex). Indeed, once we make a transmission, the lengths of all the bands that are not involved in the operation as well as the length of the winner do not change; however, the length of the loser increases exactly by the length of the winner. The reduction does not influence the vertical lengths of bands. 

So, informally, the \emph{cocycle} is responsible for things that happen with vertical lengths of bands during the application of the Rauzy induction. 

More precisely, let $B$ be a matrix of the cocycle. For $n$ steps of the non-accelerated induction that do not form yet the step of an accelerated one (and therefore the combinatorics does not change) we denote by $B(n)$ is the following matrix:
$$\begin{pmatrix}
1 & 0 & 0\\
n & 1 & 0 \\
n & 0 & 1
\end{pmatrix}$$

Thus the matrix $B$ of the cocycle is a product $B(n)$ with different $n$ and the required permutation matrices. 
Let us denote by $B_{\gamma}$ the cocycle matrix that corresponds to the path $\gamma$ in $\Gamma$.

\section{The roof function and the suspension flow}\label{3}
In this section we define the roof function associated with the cocycle and then use it to construct the suspension flow. We also prove some important estimations for the roof function that will be used later.
\subsection{The roof function}
\begin{definition}
A path $\gamma$ in the Rauzy graph is called \emph{positive}, if $B_{\gamma}$ has only strictly positive entries.
\end{definition}
\begin{lemma}
Any complete path is positive.
\end{lemma}
\begin{proof}
We start from identity matrix of the cocycle.
The fact that $\alpha$ is the winner in terms of cocycle matrix means that row with number $\alpha$ is added to the two other rows (here we always mention the original enumeration and do not care about permutations). 
If a path is complete, than each row was added to two others at least once, so all zeros disappear.
\end{proof}
The word in our coding that corresponds to a positive path is called \emph{positive}.

Let us consider a complete path $\gamma_{*}$ and the sub simplex of the parameter space that corresponds to this path $\Delta_{\gamma_{*}}$.
\begin{definition}
The roof function $r$ is the first return time to the subsimplex $\Delta_{\gamma_{*}}$:
\begin{equation}\label{roof}
r(\lambda, \tau) = -\log||(B^{T}_{\gamma_{*}})^{-1} \lambda||,
\end{equation}
where $\lambda = (a,b,c)$ is a vector of lengths and $\tau$ is a corresponding permutation.
\end{definition}

\subsection{Properties of the roof function}
Now we prove some important properties of the roof function $r$ following the ideas from \cite{AGY} and \cite{BG}.

Let $\phi$ be a function: $X\rightarrow \mathbb R$, $X$ is a space of the Markov shift. We denote by $$var_{n}(\phi)=sup\{|\phi(x)-\phi(y)|: x_{i}=y_{i}, i=1,\cdots,n\}$$ \emph{n-th variation} of $\phi$.
\begin{definition}
We say that $\phi$ has \emph{summable variations} if $$\sum_{n=2}^{\infty}var_{n}(\phi)<\infty.$$
\end{definition}
\begin{definition}
The function $\phi$ is locally H\"older continuous if there exists $C_{0}>0$ and $0<\theta<1$ such that for all $n\geq1$ $var_{n}(\phi)\leq C_{0}\theta^{n}.$
\end{definition}
\begin{lemma}
The roof function is bounded away from zero and locally H\"older continuous. In particular, it has summable variations.
\end{lemma}
\begin{proof}
First, we prove the the roof function is bounded away from zero:
$$r(\lambda,\tau)=-\log||B^{T}_{\gamma}\lambda|| =\log||R\gamma\lambda^{'}||=\log(\sum_{i=1}^{3}\alpha_{ij}\lambda^{'}_{j})\geq \log(3\cdot \sum_{j=1}^{3}\lambda^{'}_{j})=\log3,$$ because $\gamma$ is a positive path and so all $\alpha_{ij}>0.$

The statement about the H\"older property follows directly from the uniformly expanding/contraction property of the induction map that was proved in our previous paper (\cite {AHS}). 

First, the roof function is locally constant on Markov cells and can be considered as a function on the space of Markov shift. Then, for each two points $x$ and $y$ from the same cylinder of depth $n$, we can find vectors $\lambda_{x}$ and $\lambda_{y}$ such that their symbolic dynamics is described by $x$ and $y$ respectively, and $$r(x)=r(\lambda_{x},\tau_{x}),$$ $$r(y)=r(\lambda_{y},\tau_{y}).$$
Now, each of these points have a preimage ($\lambda_{x}^{'}$ and $\lambda_{y}^{'}$ respectively) with respect to the induction such that 
$$\lambda_{x}=\frac{R\lambda^{'}_{x}}{||R\lambda^{'}_{x}||}$$ and $$\lambda_{y}=\frac{R\lambda^{'}_{y}}{||R\lambda^{'}_{y}||}$$ with the same matrix $R$ because the symbolic dynamics for $x$ and $y$ coincides up to step $n$.
Here $R$ is the matrix of the accelerated induction.

Now we have to use the fact that the projectivization of the map defined by the matrix with non-negative entries is always non-expanding with respect to the Hilbert metrics and when all the coefficients are positive it is strongly contracting (see \cite{AF} or \cite{B} for details). Therefore, for some $\theta<1$ it holds that
$$|r(x)-r(y)|= \frac{\log(R_{x}\lambda_{x})}{\log(R_{y}\lambda_{y})}\leq C_{2}d_{H}(\lambda_{x},\lambda_{y})\leq C_{2}C_{1}\theta^{n},$$
where $C_{1}$ and $C_{2}$ are uniform constants  and for the last estimation we use the fact that the sum of the entries in the each row of matrices we work with is positive.
\end{proof}

\subsection{The suspension flow}
We use a standard definition of the suspension flow constructed by the shift transformation $\sigma$ and the roof function $r$ (it is an analogue of the construction from \cite{V} for the Teichm\"uller flow). The suspension flow renormalizes the length of the interval to $1$. 

Formally, the definition is as follows.
The flow $\Phi(\sigma, r)$ is defined on a space $Y=((\lambda, \tau,t)\in \sigma \times \mathbb R: 0 \leq t \leq r(\lambda, \tau))$ 
and the points $(\lambda, \tau, r(\lambda, \tau))$ and $(\sigma(\lambda, \tau),0)$ are identified. It acts in the following way: 
$$\Phi_t(\lambda,\tau,s)=(\lambda,\tau,t+s)$$ 
whenever $s+t \in [0,r(\lambda,\tau)].$

\section{Thermodynamical formalism for the Rauzy gasket}\label{4}
The main result of this section is the following:
\begin{theorem}
There exists a measure of maximal entropy for the suspension flow of the Rauzy gasket, and this measure is unique.
\end{theorem}
\noindent The same results for Teichm\"uller flow were obtained in \cite {BG} and \cite {H}.
\begin{remark}
This measure induces the measure on the Rauzy gasket that was presented in Theorem \ref{main}.
\end{remark}

The proof is based on the thermodynamical formalism for a countable Markov shift developed by O. Sarig (\cite{S0}, \cite{S1}, \cite{S2}). 

\subsection{Ruelle operator}
As in the previous section, we denote a point $(\lambda, \tau)$ of the parameter space by $x$.
Let us consider $\phi(x)=-3r(x),$ where $r(x)$ is the roof function, as a potential. 
Then we can construct a standard Ruelle operator for the Markov shift $\sigma$ based on this potential:
$$L_{\phi}f(x)=\sum_{\sigma(y)=x}e^{\phi(y)}f(y).$$
Our first goal is to prove the following theorem:
\begin{theorem}
$||L_{\phi}1||_{\infty}<+\infty.$
\end{theorem}
\begin{proof}
We start from a couple of standard lemmas that relate the roof function (as the first return time) and the measure of Markov cells that we obtain after this time $t$. Both lemmas follow directly from the fact that $e^{3r(x)}=|DT(x)|,$ where $T$ is Markov map (see \cite{AHS}, Lemma 12). 
\begin{lemma}
If $r(a)=r(b)$, then $\lambda([a])=\lambda([b]),$ where $\lambda$ is the Lebesgue measure and $[a],[b]$ are cylinders.
\end{lemma}
\begin{lemma}
If $|r(a)-r(b)|\leq1,$ then $e^{-3}\leq\frac{\lambda([a])}{\lambda([b])}\leq e^{3}.$
\end{lemma}
Now we want to evaluate $L_{\phi}1=\sum_{a\in S}e^{\phi(a)}=\sum_{N=0}^{+\infty}\sum_{s:N\leq r(s)<N+1}e^{\phi(s)},$ where $\phi=-3r.$
One can check that $$L_{\phi}1\leq \sum_{N=0}^{+\infty} Card(\{a: N\leq r(a)<N+1\})e^{-3N}.$$
Let us denote $Y(N)=\{a: N\leq r(a)<N+1\}$.

Let us recall that in \cite{AHS} we proved that the roof function $r$ has an exponential tail: there exists a positive constant $\beta>0$ such that $\int e^{r\beta}d\lambda<+\infty.$
\begin{lemma}
There exists a constant $C>0$ such that 
$Card (Y(N))\leq Ce^{(3-\beta)N},$ where $\beta$ is the constant from the definition of the exponential tail.
\end{lemma}
\begin{proof}
\begin{equation}\label{int}
\int_{\Delta}e^{r\beta}d\lambda = \sum_{N=0}^{+\infty}\Big(\sum_{a\in Y(N)}\int_{[a]}e^{r\beta}d\lambda\Big)<+\infty,
\end{equation}
so \eqref{int} implies that 
$$\sum_{a\in Y(N)}\int_{[a]}e^{r\beta}d \lambda \to 0(N \to+\infty).$$
So, starting from some moment ($N\geq N_{0}) \sum_{a\in Y(N)}\int_{[a]}e^{r\beta}d \lambda<1.$
But $\int_{[a]}e^{r\beta}d \lambda\geq \lambda[a]e^{N\beta}.$
Also, $\sum_{a\in Y(N)}\lambda([a])\geq Card(Y(N))\times e^{-3N-3}.$
So, $$Card(Y(N))\times e^{N\beta}\times e^{-3N-3}<1$$ and the statement of the lemma holds with $C=e^{3}.$
\end{proof}
Now we can finish the proof of our theorem: 
$$||L_{\phi}1||_{\infty}\leq C'\sum_{N=0}^{+\infty}Ce^{(3-\beta)N}e^{-3N},$$ where $C'$ is some constant. Then, we have a geometric series with a common ratio $e^{-\beta}<1.$ The series converges.
\end{proof}
\subsection{Existence and uniqueness of the Gibbs measure.}
The definition of the Gibbs measure can be found in \cite{S2}.
\begin{theorem}\label{77}
Let us consider the potential $\phi=-\kappa r,$ where $r$ is the roof function and $\kappa>3.$
Then for the Markov map $T$ defined on the Rauzy gasket there exists an invariant Gibbs measure with this potential, and this measure is unique.
\end{theorem}
\begin{proof}
We need one additional notation and then one more important definition.
Let us denote the n-th ergodic sum for $\phi$: $$\Phi_{n}(x)=\sum_{k=0}^{n-1}\phi(\sigma^{k}(x)).$$
Then $$Z_{m}(\phi,i)=\sum_{\sigma^{m}(x)=x, x_{0}=i}exp(\Phi_{n}(x)).$$
\begin{definition}
The Gurevich-Sarig pressure is $P_{G}=lim_{m\rightarrow\infty}\frac{\log Z_{m}}{m}.$
\end{definition}
In \cite{S1} it was proved that it does not depend on $i$ and the variational principle holds:
\begin{equation}\label{var}
sup\{h_{\mu}(\sigma)+\int\phi d\mu\}<\infty.
\end{equation}
The sup is taken for all measures such that $\int (-\phi)d\mu<\infty.$
We need two following theorems by Sarig. First, 
\begin{theorem}\label{S1}
[Sarig] If $\Sigma$ is a topologically mixing countable Markov shift and the potential $\phi$ is locally H\"older with $||L_{\phi}1||_{\infty}<+\infty,$ then $P_{G}(\phi)<+\infty,$ where $P_{G}$ is the Gurevich-Sarig pressure.
\end{theorem}
So, in our case we have the finite Gurevich-Sarig pressure. Then, we have the following:

\begin{theorem} \label{S2}
[Sarig] Assume that $\phi$ has summable variations. Then $\phi$ admits a unique $\sigma$-invariant Gibbs measure if and only if $\sigma$ satisfies BIP property and the Gurevich-Sarig pressure is finite.
\end{theorem}

For proofs see \cite{S1}.
So, we can apply this theorem to the ``if'' side and obtain that there exists unique Gibbs measure that is invariant for our Markov shift. 
This ends the proof of Theorem \ref{77}.
\end{proof}
\subsection{Existence and uniqueness for the equilibrium measure}
For the definition of the equilibrium measure see \cite{Bo}.
\begin{theorem}
Let us consider the potential $\phi=-\kappa r,$ where $r$ is the roof function. Then for the Markov map $T$ defined of the Rauzy gasket there exists an invariant equilibrium measure, and this measure is unique.
\end{theorem}
\begin{proof}
In accordance with Corollary 2 from \cite{S2} our potential $\phi$ is positive recurrent (see \cite{S2} for precise terminology) and there exist $\lambda>0, h>0$ and $\nu$ - conservative Borel measure such that the following conditions hold:
\begin{itemize}
\item $L_{\phi}h=\lambda h$;
\item $L^{*}_{\phi}\nu=\lambda\nu$;
\item $\int hd\nu<\infty$;
\item $h$ is bounded away from zero and infinity and $\nu$ is finite.
\end{itemize}

So, the equilibrium measure exists, and the uniqueness follows from \cite{BS}.
\end{proof}

\subsection{Existence and uniqueness of the measure of maximal entropy for the flow}
Let us consider the family of potentials $\phi(\kappa)=-\kappa r$ and the value of the corresponding Ruelle operator for each of them in point $f=1$: $$L_{\phi(\kappa)}1(x)=\sum_{y: \sigma(y)=x}e^{-\kappa r(y)}.$$
We need the following technical 

\begin{lemma}
The following limit holds as $n\rightarrow\infty:$
$$\frac{\log L^{n}_{\phi(\kappa)}(1)(x)}{n}\rightarrow P_{\kappa}.$$
\end{lemma}
\begin{proof}
This fact is a direct application of Theorem 1 from \cite{S0}. Moreover, from this result by Sarig we know that convergence holds for every $\kappa$ but the limit can be equal to $+\infty$. In the case of $\kappa>3$ 
due to Theorem \ref{77} the limit is bounded above by some constant.
\end{proof}

Now, the fact that the roof function is bounded away from 0 implies that $P_{\kappa}\to -\infty$ when $\kappa \to\infty.$ On the other hand, $P_{0}>0$ or $P_{0}=+\infty,$ and in both situations $P_{\kappa}$ is a decreasing continuous function of $\kappa$.
So we have that
\begin{lemma}
There exists $\kappa_{0}$ such that $P_{\kappa_{0}}=0.$
\end{lemma}
\begin{theorem}
The Gibbs measure that corresponds to the potential $\phi=-\kappa_{0}r$ is the measure of maximal entropy for the suspension flow, and the measure of maximal entropy is unique.
\end{theorem}
\begin{proof}
Let us denote the Gibbs measure that corresponds to the potential $\phi_{\kappa_{0}}=-\kappa_{0}r$ by $\mu_{0}$ (this measure exists, and it is unique).
First, note that any invariant measure for the suspension flow can be associated with the invariant measure for the transformation, and vice versa (see \cite{BG} for the formula).
We want to prove that the measure $\hat \mu_{0}$ for the suspention flow that is associated with $\mu_{0}$ is the measure of maximal entropy. 
As Sarig proved, $P_{\kappa_{0}}=sup_{\mu}\{h_{\mu}(\sigma)-\int_{X}\kappa_{0}rd\mu\}$, and $\mu_{0}$ is exactly the measure for which $sup$ value is achieved. 
But in our case $P_{\kappa_{0}}=0$. Therefore, 
$$h_{\mu_{0}}(\sigma)-\int_{X}\kappa_{0}rd\mu_{0}=0$$
and for any other invariant $\mu$
$$h_{\mu}(\sigma)-\int_{X}\kappa_{0}rd\mu<0.$$

\noindent Now one can apply Abramov formula (\cite{AR}) to check that $h_{\hat\mu_{0}}(\Phi)=\frac{h_{\mu_{0}}(\sigma)}{\int_{X} rd\mu_{0}}=\kappa_{0}$ and for any other $\hat\mu$ $h_{\hat\mu}(\Phi)<\kappa_{0}$ ($\Phi$ here is the flow).

\noindent Uniqueness follows from \cite{BS}.
\end{proof}

\section{Lyapunov exponents}\label{5}
\subsection{The Lyapunov exponents for the standard cocycle}
In this section we introduce the Lyapunov exponents for the cocycle $B$. Our main tool is the multiplicative ergodic theorem by V. Oseledets (\cite{O}, Theorem 2 and Theorem 4). 
Let us check that all the conditions of this theorem are satisfied in our case.

\noindent First, the suspension flow $\Phi$ preserves the measure of maximal entropy $\mu=\hat\mu_{0}$ and is ergodic with respect to this measure. The last fact is a direct corollary of the results by Sarig (\cite{Sa}, Section 4.3.3, Theorem 4.7) on the ergodicity (more precisely, the strongly mixing property) of the measure $\mu$.

\noindent Now, the cocycle $B$ constructed in section 3.6 is a measurable cocycle over the flow because the cocycle is locally constant. 

\noindent Finally, $B$ and $B^{-1}$ are $\log$ - integrable with respect to the considered measure:
\begin{lemma}
$$\int\log^{+}||B||d\mu<+\infty,$$
$$\int\log^{+}||(B^{-1})^{T}||d\mu<+\infty.$$
\end{lemma}
\begin{proof}
Due to Theorem \ref{S1} we know that the Gurevich - Sarig pressure $P_{G}$ is finite as well as the entropy $h_{\mu}$, and also \eqref{var} holds.
It implies that $\int \phi d\mu<\infty$ where $\phi$ is the potential: $\phi = -\kappa_{0} r$. Therefore $\int rd\mu<\infty$ and so by (\ref{roof}) $\int\log^{+}||(B^T)^{-1}||d\mu<+\infty.$ Since $B\in SL(3,\mathbb R),$ the last statement implies that the cocycle $B$ is also log-integrable. 
\end{proof}

\noindent Therefore, we can apply Oseledets theorem: 
\begin{theorem}
There exist numbers $\lambda_{1}\geq \lambda_{2} \geq\lambda_{3}$ such that for almost all points $x$ from the Rauzy gasket there exists a filtration $\mathbb R^{3}=E_{1}(x)\supseteq E_{2}(x)\supseteq E_{3}(x)\supseteq E_{4}(x)=0$ and
for every $v \in E_{i}\setminus E_{i+1}$
$$\lim_{N\to\infty}\frac{\log||B(x)v||}{N}=\lambda_{i},$$ where $B=B^{(N)}$ is the matrix of the cocycle comprising $N$ blocks $B(n_i)$ (possible with permutations). 
\end{theorem}
These $\lambda_{i}$ are called the \emph{Lyapunov exponents} of the cocycle $B$.

\subsection{Lyapunov exponents for the cocycle with orientation}\label{6.2}
In order to describe the behavior of the chaotic plane sections we need to introduce another cocycle, say A, which is responsible for changes of the basis in homology. The main difference between $B$ and the cocycle $A$ is that $A$ contains the information about orientation of the bands. 

More precisely, we have the following. Let us fix the orientation on the bands of the original complex (see Figure \ref{C}). Then, the application of the Rauzy induction comprises two steps: transmission (when two bands are transmitted along one which is the largest) and reduction (when the largest band is cut). The reduction preserves the vertical lengths of the bands and the orientation of all bands while the transmission acts 
in the following way: each band that was transmitted is replaced by a long band that is the union of the original one (with the same orientation) and the largest band (but with the opposite orientation); the orientation of the largest band does not change. One can check that up to the permutation of lines one [accelerated] step of this procedure is described by 
$$A(n)=\begin{pmatrix}
1 & -n & -n\\
0 & 1 & 0 \\
0 & 0 & 1
\end{pmatrix}.$$
Like in a case of the cocycle $B$, in general $A$ is constructed from these blocks $A(n)$ with appropriate $n$, and permutation matrices. 

Now we establish the connection between the matrices of these two cocycles: 
\begin{lemma}\label{transp}
Let us fix the path $\gamma$ on the Rauzy graph and denote the corresponding matrices of the cocycles by $B_{\gamma}$ and $A_{\gamma},$ respectively.
Then $(B^T_{\gamma})^{-1}=A_{\gamma}.$
\end{lemma}
\begin{proof}
For every $n$, $B(n)^T$=$A(n)^{-1}$. Then, $B_{\gamma}=B_{n_1}S_{n_1}\cdots B_{n_i}S_{n_i},$ where $S_{n_j}$ is a permutation matrix.
So 
$$
(B^T_{\gamma})^{-1}=(S^T_{n_i}B^T_{n_i}\cdots S^T_{n_1}B^T_{n_1})^{-1}=(B^T_{n_1})^{-1}S_{n_1}\cdots (B^T_{n_i})^{-1}S_{n_i}=$$
$$=A_{n_1}S_{n_1}\cdots A_{n_i}S_{n_i} = A_{\gamma}.
$$
Here we used that the permutation matrices are orthogonal.
\end{proof}
Let us now order the Lyapunov exponents of both cocycles from the largest to the smallest. Then, there is an obvious 
\begin{corollary}
The following relation between the Lyapunov exponents of the cocycles $A$ and $B$ holds:
$$\lambda_i=\lambda_i(B)=-\lambda_{3-i}(A)$$ for $i=1,2,3.$
\end{corollary}

\section{Lyapunov exponents are responsible for the diffusion rate}\label{6}
The main result of this section is the following
\begin{theorem}\label{rate}
For almost every point $(a,b,c)$ from the Rauzy gasket (with respect to the measure $\mu$) for every point $x\in \hat M$ 
\begin{equation}\label{traj}
\limsup_{t\to\infty} \frac{d(x,x_{t})}{\log t}=-\frac{\lambda_{3}}{\lambda_{1}},
\end{equation}
where $x_t$ is the point at distance $t$ from the point $x$ along the leaf of the foliation $F=F(a,b,c)$.
\end{theorem}
\begin{proof}
Lemma \ref{dynlem} implies that it is enough to evaluate the diffusion rate of the vertical section of $\hat \Sigma^{'},$ where $\Sigma^{'}$ is the complex and $\hat\Sigma^{'}$ is its abelian cover. 

Now, we explain why this diffusion rate is controlled by the cocycle $A$. The main idea is as follow: the suspension construction provides natural basis for $H_{1}(M_1, \mathbb Z)$ (the same statement holds for the homologies of the surface and interval exchange transformations, see, for instance, \cite{Y}, Chapter 4.5). In order to see it, one has to choose as an element of the basis a cycle that connects the centers of intervals $A_i$ and $B_i$ along the bands and is closed by the part of the support interval. Then, one can check that the cocycle acts in the homologies of $M_1$, and this action induces the action of a family of  automorphisms of the free group on the fundamental group $\pi_1$.

More precisely, the cocycle $A$ contains the information on the induction; if we enumerate the bands of the corresponding widths (with orientation) by $\tilde a, \tilde b, \tilde c,$ respectively, the induction works in the following way (see subsection \ref{6.2}): 
$$\tilde a'=\tilde a;$$
$$\tilde b'=\tilde b{\tilde a}^{-1};$$
$$\tilde c'=\tilde c{\tilde a}^{-1}.$$

Vertical sections of $\hat \Sigma$ contain the following blocks (see Figure \ref{complex}):
\begin{enumerate}
\item horizontal: $\tilde b{\tilde a}^{-1}, \tilde a{\tilde b}^{-1}$;
\item vertical: $\tilde c{\tilde a}^{-1}, \tilde a{\tilde c}^{-1};$
\item vertical: $\tilde b{\tilde c}^{-1}, \tilde c{\tilde b}^{-1}.$
\end{enumerate}

So, every vertical section (equivalently, the trajectory) is coded by $\tilde a, \tilde b, \tilde c$ and is equivariant under the induction and thus can be coded also by $\tilde a', \tilde b', \tilde c'$. 
The only obstacle to use this description directly is the existence of backtracks. 

It is easy to see that the backtracks can appear only in the third combination (with $\tilde b$ and $\tilde c$ with different signs). Since, say, the trajectory that is coded by $\tilde b'\tilde c'^{-1}$ is a backtrack, it does not contribute to the diffusion rate; so one has to check that it does not contribute to the right part of (\ref{traj}) (in other words, that it is not noticed by the cocycle $A$).
Indeed, since the vector $(0,1,-1)$ is invariant for the matrix of the cocycle, it does not make any contribution.

Now we continue to prove the theorem. 

Let us consider the vertical section of the complex $\hat\Sigma'$ with the natural time $t$ on this curve. We denote by  $\hat\psi(t)$ the part of this curve between time $0$ and $t$. The projection of  $\hat\psi(t)$ to the original complex $\Sigma$  is denoted by $\psi(t)$. 

Let us also consider for every $k \in \mathbb N$ the subinterval $D_k$ that is the support interval of the system of isometries obtained from $S$ after $k$ iterations of the accelerated Rauzy induction. Let $n=n(t)$ be the largest $k$ such that $\psi(t)$ intersects $D_k$ at least twice.

Then the following holds:
\begin{lemma}\label{denom}
Let $n$ and $t$ be as described above. Then
$$\lambda_{1}n\sim\log t.$$
\end{lemma}
\begin{proof}
The proof repeats the proof of the same statement in \cite{Z} (see Chapter 4.9, in particular Lemma 4).
\end{proof}

Now, we prove the {\bf upper bound} estimation. In this part we mainly follow the strategy from \cite{Z}, \cite{DHL} and \cite{F}. The main idea is to decompose any trajectory into parts whose deviation is understandable by Oseledets theorem. 

We take the point $x$ on the curve $\hat\psi$ and consider the vector $v_{t}=x_{t}-x\in \mathbb R^{3}.$ Since the directions of the bands and the support interval of complex $\Sigma'$ are parallel to the coordinate lines and the plane is parallel to $x_2=0$, one can check that the vector $v_t$ is a linear combination of images by the cocycle $A$ of two basis vectors $e_1$ and $e_3$, where $e_1=(1,0,0)$ and $e_3=(0,0,1)$.

We denote by $A^{(i)}$ the matrix that is a product of $i$ blocks $A(n_j), j=1,\cdots,i$ and permutation matrices.

More precisely, we have the following representation: 
\begin{lemma}
\begin{equation}
\label{decomposition}
v_{t}=\sum_{i=1}^{n}\sum_{j=1,3}a^{j}_{i}v^{i}_{j},
\end{equation}
where $v_{i}^{j}=A^{(i)}e_{j}$, $j=1,3$ and $a^{j}_{i}$ are non-negative integers with subexponential growth.
\end{lemma}
\begin{proof}
The proof is similar to the calculations done in \cite{DHL} (Chapter 5.2, Lemma 8), \cite{Z}(Chapter 4.9), \cite {Zo97} (Proposition 8) and \cite{F} (Lemma 9.4).
\end{proof}

Using the subexponential growth of $a^{j}_{i}$ and Oseledets theorem for the cocycle $A$, for every $\epsilon>0$ we have
\begin{equation}\label{nom}
\log d(x,x_{t})\leq log ||v_{t}||\leq \log(\sum a^{j}_{i}v^{i}_{j})\leq -n(\lambda_{3}+\epsilon).
\end{equation}

Combining \eqref{nom} with Lemma \ref{denom} we get the upper bound and turn to the {\bf lower bound} estimations. 
We start from the following technical 
\begin{lemma}
Let us consider $E=<e_{1},e_{3}>$ and the direct sum induced by Oseledets decomposition for the cocycle $A$: $\mathbb R^{3}=\tilde E_{1}\oplus \tilde E_{2}\oplus \tilde E_{3}$. 
There exists $f\in E$ such that the projection of $f$ on $\tilde E_{3}$ is not equal to zero for almost every $(a,b,c)$ from the Rauzy gasket.
\end{lemma}
\begin{remark}
One can assume that the spectrum is simple (otherwise the lemma is trivial).
\end{remark}
\begin{proof}
We need to check that $E\ne\tilde E_{1}\oplus \tilde E_{2}.$ By contradiction, the space $E$ would be invariant with respect to the cocycle $A'$ induced by the first return map on a subset of positive measure. But the matrix of the inverse cocycle is obtained by the additional acceleration from the matrix of the cocycle $B^T$ and thus after a sufficient number of iterations of the Rauzy induction has only strictly positive coefficients. It implies that the space $E$ can not be an invariant space for the cocycle $A'$.
\end{proof}

Then, since there exists such $f$, one of the following statements holds:
\begin{equation}\label{case1}
\frac{\log||A^{(k)}e_1||}{k}\to \lambda_1(A)=-\lambda_3
\end{equation} or 
\begin{equation}\label{case2}
\frac{\log||A^{(k)}e_3||}{k}\to \lambda_1(A)=-\lambda_3.
\end{equation}
Therefore, one can apply the standard approach from \cite{Z} (Chapter 4.10) to get the lower bound. Namely, we use again the decomposition \eqref{decomposition}.  
If the required inequality does not hold when the trajectory enters the interval $D_n$ for the first time, we continue to follow the trajectory. Since we have \eqref{case1} or \eqref{case2}, there exists the first moment when the required inequality holds for some part of the trajectory (the choice of the part depends on which of two alternatives holds). It implies the statement about the lower bound.
See \cite{DHL}(Chapter 5.3) and \cite{Z} (Proposition 3) for technical details.
\end{proof}

\section{Final estimations}
\label{7}
\subsection{Pisot property}
In this subsection we prove that the matrix of the cocycle satisfies the so called \emph{Pisot property}. 
\begin{definition}
A matrix is called \emph{Pisot} if it has only one dominant eigenvalue (i.e. an eigenvalue of maximum modulus), and all eigenvalues different from the dominant one have norm less than one.
\end{definition}

The key ingredient of the proof is the strong connection between the cocycle we work with and so called \emph{fully subtractive algorithm} that was first pointed out in \cite{AS}. 
The fully subtractive algorithm is defined on the positive cone $\mathbb R^3_{\ge 0}$, it subtracts the smallest number from the two others, i.e., it is given by the map $S: \mathbb R^3_{\ge 0} \mapsto \mathbb R^3_{\ge 0}$ given by
\begin{equation*}
\tilde S: (x_1, x_2, x_3)\mapsto
 \begin{cases}
   (x_1, x_2-x_1, x_3-x_1) &\text{if $x_1\le x_2, x_1 \le x_3$}\\
   (x_1- x_2, x_2, x_3-x_2) &\text{if $x_2\le x_1, x_2 \le x_3$}\\
   (x_1- x_3, x_2-x_3, x_3) &\text{if $x_3\le x_1, x_3 \le x_2$}
 \end{cases}
\end{equation*}
One can easily check that the matrix $B(n)$ coincides with the matrix of the fully subtractive algorithm after $n$ iterations. Now, Pisot property for the fully subtractive algorithm was proved by Avila and V. Delecroix (see \cite{AD}, Section 2).  It follows that 
\begin{lemma}
The matrix of the cocycle $B$ is Pisot.
\end{lemma}

Lemma 6 from \cite{AD} and the last lemma imply that the Lyapunov exponents $\lambda_{1}>0>\lambda_{2}\geq \lambda_{3}$ satisfy and obviously we have $$\lambda_{1}+\lambda_{2}+\lambda_{3}=0.$$ It follows that

\begin{corollary}\label{right}
The Lyapunov exponent for the suspension flow satisfies $$\lambda = -\frac{\lambda_{3}}{\lambda_{1}}<1.$$
\end{corollary}  

\subsection{Simplicity of the spectrum}
In this section we prove that the Lyapunov spectrum for a special system of isometries of thin type is typically simple. 
In order to check it, we have first to verify that our measure $\mu$ and the cocycle $B$ over the transformation $\sigma$ satisfy the following conditions:
\begin{itemize}
\item the transformation is a countable Markov shift; 
\item $\mu$ has \emph{bounded distortion}:
$$\frac{1}{\hat C}\leq\frac{\mu[i_{-m}\cdots i_{-1} i_{0}\cdots i_{n}]}{\mu[i_{-m}\cdots i_{-1}]\mu[i_{0}\cdots i_{n}]}\leq{\hat C}$$ for some uniform constant $\hat C>0$;
\item the matrix of the cocycle is locally constant. 
\end{itemize}

The first and the third conditions obviously hold for our construction. The next lemma is responsible for the second one: 
\begin{lemma}
The measure of maximal entropy $\mu$ satisfies the bounded distortion property. 
\end{lemma}
\begin{proof}
The statement actually holds for all Gibbs measures (in particular, for the measure of maximal entropy). Let us choose one with a potential $-\kappa r.$ 
Then, there exist two uniform constants $P$ (the pressure) and $Q$ such that for every cylinder $[i_{0}\cdots i_{n-1}]$ and every point $x$ from this cylinder we have
$$\frac{1}{Q}\leq\frac{\mu[i_{0}\cdots i_{n-1}]}{e^{\Phi_{n}(x)-nP}}\leq Q.$$
Now one can check that if $x\in [i_{-m}\cdots i_{-1}\cdots i_{0}\cdots i_{n}]$
$$\frac{\mu([i_{-m}\cdots i_{-1} i_{0}\cdots i_{n}])}{\mu[i_{-m}\cdots i_{-1}]\mu[i_{0}\cdots i_{n}]}\leq \frac{Q^{3}e^{\Phi_{n+m}(x)-(n+m)P}}{e^{\Phi_{n}(x)-nP}e^{\Phi_{m}(\sigma^n(x))-mP}}=\frac{Q^{3}}{e^{P}}\leq Q^{3}e^{P}.$$

The same estimation can be done in the opposite direction as well:
$$\frac{\mu([i_{-m}\cdots i_{-1} i_{0}\cdots i_{n}])}{\mu[i_{-m}\cdots i_{-1}]\mu[i_{0}\cdots i_{n}]}\geq Q^{3}e^{P}.$$

So, the condition holds with $\hat C=Q^{3}e^{P}.$
\end{proof}
Now we apply the Galois-theoretical criterium of the simplicity of Lyapunov spectra from \cite{MMY} (Theorem 2.17). This criterium develops the idea suggested in \cite{AV}.
We have to provide first the Galois-pinching matrix (see Definition 2.12 from \cite{MMY}) for the cocycle $B$. In accordance with \cite{MMY}, the matrix of the cocycle is \emph{Galois-pinching} if its characteristic polynomial is irreducible over $\mathbb Q$, has only real roots, and its Galois group is largest possible (see Chapter 4.1 and in particular Definition 4.1 in \cite{MMY}). We work with the cocycle without orientation because Lemma \ref{transp} implies that all the properties of spectrum of the cocycle $A$ are the same.
\begin{lemma}\label{pinch}
The following matrix $B_{1}$ is Galois-pinching for $B$:
$$\begin{pmatrix}
12 & 6 & 5\\
11 & 6 & 5 \\
2 & 1 & 1
\end{pmatrix}.$$
\end{lemma}
\begin{remark}
The matrix $B_{1}$ corresponds to the following loop on the Rauzy graph: 
$(1,2,3)\rightarrow (2,3,1)\rightarrow (3,1,2)\rightarrow (1,2,3)$ with the following numbers of simple iterations in each accelerated iteration: 1 (for the first arrow), 1 (for the second one), 5 (for the last one). 
\end{remark}

\begin{proof}
One can check that :
\begin{itemize}
\item the characteristic polynomial is $P_{B_{1}}(\lambda)=\lambda^{3}-19\lambda^{2}+9\lambda-1;$
\item  $P_{B_{1}}(\lambda)$ is irreducible since the first and the last coefficients of $P_{B_{1}}$ are equal to $1$ and $-1$;
\item all the roots are real since the discriminant $\Delta_1>0$;
\item The Galois group is isomorphic to $S_{3}$ since $\Delta_1=1940$ is not a square of any rational number.
\end{itemize}
\end{proof}

Now we need a matrix $B_{2}$ that is twisting with respect to $B_{1}$ (see Chapter 4.2 in \cite{MMY}. especially Theorem 4.6). Following \cite{MMY}, the matrix $X$ is \emph{twisting} with respect to the pinching matrix $Y$ if it is also pinching, some natural irreducibility holds and the splitting field of its characteristic polynomial is disjoint from the splitting field of the characteristic polynomial of $Y$.
\begin{lemma}\label{twist}
The following matrix $B_{2}$ of the cocycle is twisting with respect to $B_{1}$:
$$\begin{pmatrix}
10 & 5 & 4\\
9 & 5 & 4 \\
2 & 1 & 1
\end{pmatrix}.$$
\end{lemma}
\begin{remark}
The matrix $B_{2}$ corresponds to the same loop on the Rauzy graph but with different numbers of waiting time in each vertex (for $B_{2}$ it is $(1,1,4)$).
\end{remark}

\begin{proof}
First, one needs to check that $B_{2}$ is also pinching. It follows from the fact that $P_{B_{2}}(\lambda)=\lambda^{3}-16\lambda^{2}+8\lambda-1$ and $\Delta_2=229.$ 
It also implies that two matrices $B_1$ and $B_2$ identify pairwise disjoint fields. Thus $B_2$ is twisting with respect to $B_1$.
\end{proof}
Theorem 5.4 from \cite{MMY} together with Lemmas \ref{pinch} and \ref{twist} implies the following
\begin{theorem}
The Lyapunov exponents of the cocycle are pairwise different: $\lambda_{1}>0>\lambda_{2}>\lambda_{3}$.
\end{theorem}
\begin{corollary}\label{left}
The following inequality holds for the Lyapunov exponent of the suspension flow: 
$$\lambda = -\frac{\lambda_{3}}{\lambda_{1}}>\frac{1}{2}.$$
\end{corollary}

Now Theorem \ref{main} follows directly from Theorem \ref{rate} and Corollaries \ref{left} and \ref{right}.

\begin{corollary}
For almost every plane section in chaotic case there exists a leading direction but the deviation from it is unbounded.
\end{corollary}

\end{document}